\documentclass[10pt]{article}
\usepackage{geometry}
\usepackage{amssymb,amsmath,amsfonts,amsthm,mathtools,enumerate,color}
\usepackage{mathrsfs}
\usepackage[colorlinks,citecolor=blue,urlcolor=blue]{hyperref}
\usepackage{tikz}
\usetikzlibrary{arrows,positioning,shapes.geometric}

\DeclareMathOperator*{\argmin}{arg\,min}
\usepackage{algorithm,algorithmic}

\graphicspath{ {./figures/} }
\usepackage{float}
\usepackage{subcaption}

\usepackage{indentfirst}
\usepackage{multicol}
\usepackage{url}
\usepackage[outdir=./]{epstopdf} 

\usepackage{hyperref}
\hypersetup{
    colorlinks=true, 
    linktoc=all,     
    linkcolor=blue,  
}
\allowdisplaybreaks

\theoremstyle{plain}

\newtheorem{theorem}{Theorem}[section]
\newtheorem{lemma}{Lemma}[section]

\newtheorem*{lemma*}{Lemma}

\newtheorem*{theorem*}{Theorem}
\newtheorem{proposition}{Proposition}[section]

\theoremstyle{definition}
\newtheorem{definition}{Definition}[section]

\newtheorem{remark}{Remark}[section]

\newtheorem*{remark*}{Remark}
\newtheorem*{notation*}{Notation}
\newtheorem{Assumption}{H.\!\!}

\newcommand{\NN}{{\mathbb N}}

\newcommand{\RR}{{\mathbb R}}

\newcommand{\pref}[1]{(p\ref{#1})}

\title{Well-posedness and numerical schemes for one-dimensional McKean-Vlasov equations and interacting particle systems with discontinuous drift}

\author{Gunther Leobacher\footnote{Institute of Mathematics and Scientific Computing, University of Graz. 
Heinrichstraße 36, 8010 Graz, Austria, E-mail: gunther.leobacher@uni-graz.at}, Christoph Reisinger\footnote{Mathematical Institute, University of Oxford, Andrew Wiles Building, Woodstock Road, Oxford, OX2 6GG, UK, E-mail: christoph.reisinger@maths.ox.ac.uk}, and Wolfgang Stockinger\footnote{Mathematical Institute, University of Oxford, Andrew Wiles Building, Woodstock Road, Oxford, OX2 6GG, UK, E-mail: wolfgang.stockinger@maths.ox.ac.uk}}

\begin{document}

\maketitle

\begin{abstract}
\noindent
In this paper, we first establish well-posedness results for one-dimensional McKean-Vlasov stochastic differential equations (SDEs) and related particle systems with a measure-dependent drift coefficient that is discontinuous in the spatial component, and a diffusion coefficient which is a Lipschitz function of the state only.
We only require a fairly mild condition on the diffusion coefficient, namely to be non-zero in a point of discontinuity of the drift, while we need to impose certain structural assumptions on the measure-dependence of the drift.
Second, we study Euler-Maruyama type schemes for the particle system to approximate the solution of the 
one-dimensional McKean-Vlasov SDE. 
Here, we will prove strong convergence results in terms of the number of time-steps and number of particles.
Due to the discontinuity of the drift, the convergence analysis is non-standard and the usual strong convergence order $1/2$ known for the Lipschitz case cannot be recovered for all presented schemes.  
\end{abstract}


\section{Introduction}
In this article, we study the existence and uniqueness of strong solutions for classes of McKean--Vlasov SDEs, where the drift exhibits a discontinuity in the spatial component.
We also provide time-stepping schemes of Euler--Maruyama type, for which we show strong convergence of a certain rate. 

A McKean--Vlasov equation (introduced in \cite{MK,MK2}) for a $d$-dimensional process $X=(X_t)_{t \in [0,T]}$, with a given finite time-horizon $T >0$, is an SDE where the underlying coefficients depend on the current state $X_t$ and, additionally, on the law of $X_t$. We consider more specifically the one-dimensional equation of the form
\begin{equation}\label{McKeanLimit}
    \mathrm{d}X_t =b(X_t, \mathcal{L}_{X_t}) \, \mathrm{d}t + \sigma(X_t) \, \mathrm{d}W_t, \quad X_0 = \xi \in L_2^{0}(\mathbb{R}),
\end{equation}
where $L_2^{0}(\mathbb{R})$ denotes the space of real-valued, $\mathcal{F}_0$-measurable random variables with second finite moments, $(W_t)_{t \in [0,T]}$ is a one-dimensional standard Brownian motion and $\mathcal{L}_{X_t}$ denotes the marginal law of the process $X$ at time $t \in [0,T]$. In particular, we are concerned with the well-posedness of equations of the form (\ref{McKeanLimit}), where $\RR \ni x \mapsto b(x,\mu)$ is discontinuous in zero, and piecewise Lipschitz on the subintervals $(-\infty,0)$ and $(0,\infty)$. Concerning the measure component of the drift, we will require global Lipschitz continuity with respect to the Wasserstein distance with quadratic cost denoted by $\mathcal{W}_2$ (see below for a precise definition). The diffusion term will be only state dependent and globally Lipschitz continuous. Our setting contrasts with the standard case with globally Lipschitz continuous coefficients, which is well-studied in the literature, both from an analytic and numerical perspective, e.g., in \cite{MEL}, \cite{AS} and \cite{BO}, respectively.

The study of SDEs and McKean--Vlasov equations with discontinuous drift is motivated by such models
in biology (see, e.g., \cite{FPZ}) and financial mathematics (see, e.g., Atlas models in equity markets \cite{KAR,JOU2} and dividend maximisation problems \cite{MICH}). 
Further, in stochastic control a discontinuous control can lead to equations with discontinuous drift (see \cite{MICH2}).
In the context of stochastic $N$-player games, non-smooth cost functions (such as the $\ell_1$-regularisation) or constraints on the size of the control process can result in discontinuous controls (bang-bang type optimal controls) and hence will give controlled state dynamics with discontinuous drift, as in \cite{GUO}.

We start our literature review with some key references on standard SDEs with irregular and discontinuous drift, namely \cite{AKZ,AVDIS,sz2016a,sz2016b}, and then proceed to discuss some recent articles on McKean--Vlasov SDEs with non-Lipschitz drift. Zvonkin \cite{AKZ} (for one-dimensional SDEs) and Veretennikov \cite{AVDIS} (for the multi-dimensional setting) prove the existence of a unique strong solution for an SDE where the drift is assumed to be measurable and bounded, but the diffusion coefficient $\sigma$ needs to satisfy rather strong assumptions, namely that it is bounded and uniformly elliptic, i.e., there is a $\lambda >0$ such that for all $x \in \mathbb{R}^{d}$ and all $v \in \mathbb{R}^{d}$, we have $v^{\top}\sigma(x)\sigma(x)^{\top}v \geq \lambda v^{\top}v$. An interesting addition to the aforementioned results in the case where the diffusion is not uniformly elliptic was established in the one-dimensional case in \cite{sz2016a}. The authors assume the drift coefficient to be piecewise Lipschitz and $\sigma$ to be globally Lipschitz with $\sigma(\eta) \neq 0$ for each of finitely many points of discontinuity $\eta$ of the drift. 
This condition guarantees that the process does not spend a positive amount of time in the singularity. Under these assumptions, by explicitly constructing a transformation that removes the singularities, the existence of a unique strong solution can be proven and a numerical procedure for solving this class of SDEs can be constructed.

The main contribution of \cite{sz2016b} is the extension of the one-dimensional case to the multi-dimensional setting under the assumption of piecewise Lipschitz continuity of the drift. In \cite{sz2016b}, the authors introduce a meaningful concept of piecewise Lipschitz continuity in higher dimensions, which is based on the notion of the so-called intrinsic metric. As already indicated by the one-dimensional case, there needs to be an intricate connection between the geometry of the set of discontinuities and the diffusion coefficient. We note that the exceptional set of singularities, denoted by $\Theta$, is assumed to be a $\mathcal{C}^{4}$ hypersurface and for the diffusion part one requires the following: There exists a constant $C>0$ such that $| \sigma(\eta)^{\top} n(\eta) | \geq C$ for all $\eta \in \Theta$, where $n(\eta)$ is orthogonal to the tangent space of $\Theta$ in $\eta$ and $| n(\eta)| =1$. Under these assumptions (and some additional technical conditions on the coefficients and on the geometry of $\Theta$) the existence of a unique strong solution for multi-dimensional SDEs with piecewise Lipschitz continuous drift can be proven. 

Moving on to McKean--Vlasov equations, the existence and uniqueness theory for strong solutions of such SDEs with coefficients of linear growth and Lipschitz type conditions (with respect to the state and the measure component) is well-established (see, e.g., \cite{AS,RC}). 
More general existence/uniqueness results for weak and strong solutions of McKean--Vlasov SDEs can be found in \cite{MVA,LACK,BBP}. The article \cite{BBP} is concerned with the weak and strong existence/uniqueness of one-dimensional equations with additive noise, where the drift is assumed to be measurable, continuous in the measure component with respect to the Monge-Kantorovich metric and further satisfies a linear growth condition. In \cite{MVA} and \cite{LACK}, a $d$-dimensional setting is considered, where the drift is assumed to be bounded, measureable (and possibly path-dependent) and Lipschitz continuous in the measure component with respect to the total variation distance. The diffusion is non-degenerate and independent of the measure. Under these assumptions (and some technical conditions) weak existence and uniqueness is proven. 

For further recent existence and uniqueness results for strong and weak solutions of McKean--Vlasov SDEs, including results concerning standard Lipschitz assumptions on the coefficients, we refer to \cite{HSS,JLHM,C1,RAY,RAY2,ROCK,FKM} and the references given therein.   

The numerical analysis of SDEs with discontinuous drifts has received significant attention over the last few years, see, e.g., \cite{sz2016a,sz2016b,sz2016c,NMSLS,MGLY219,MGLY19,HNK,NGO,CHA,DISC} and the references therein for well-posedness results and for strong and weak convergence rates of numerical schemes. 

In particular, in \cite{sz2016a} (for the one-dimensional case) and in \cite{sz2016b} (for the multi-dimensional setting) the standard strong convergence rate of order $1/2$ for a method derived from the Euler--Maruyama scheme was proven. However, the applicability of these schemes is limited as they require the explicit knowledge of a transformation (and its inverse) to map the SDE with discontinuous coefficients into one with Lipschitz continuous coefficients. In \cite{sz2016c}, an Euler--Maruyama scheme without the aforementioned transformation is introduced (in a multi-dimensional setting). While this scheme is easier to apply, the authors only show a strong convergence rate of order $1/4 - \varepsilon$ for any $\varepsilon>0$, imposing also the stronger assumption of boundedness for both coefficients of the underlying SDE. The central idea of \cite{sz2016c} is to quantify the probability that a multi-dimensional process is in a small neighbourhood of the set of discontinuities, using an occupation time formula. In the one-dimensional case, with coefficients of linear growth, these techniques were refined in \cite{MGLY19} and the expected strong convergence rate of order $1/2$ was recovered. Other recent works concerned with the numerical approximation of SDEs with discontinuous drifts include \cite{NMSLS,MGLY219}, where a higher order scheme and an adaptive time-stepping scheme were introduced, respectively. In \cite{DISC} a numerical scheme for classical one-dimensional diffusion processes generated by a differential operator involving discontinuous coefficients is presented. As the generator is non-local for McKean--Vlasov equations it seems a challenging problem to use these techniques in our framework.   

The simulation of McKean--Vlasov SDEs typically involves two steps: First, at each time $t$, the true measure $\mathcal{L}_{X_t}$ is approximated by the empirical measure
\begin{equation*}
 \mu_t^{\boldsymbol{X}^{N}}(\mathrm{d}x) := \frac{1}{N}\sum_{j=1}^{N} \delta_{X_t^{j,N}}(\mathrm{d}x),
\end{equation*}    
where $\delta_{x}$ denotes the Dirac measure at point $x$ and $(\boldsymbol{X}^{N}_t)_{ t \in [0,T]} = (X_t^{1,N}, \ldots, X_t^{N,N})_{t \in [0,T]}^{\top}$, an interacting particle system, is the solution to the $\mathbb{R}^{dN}$-dimensional SDE with components
\begin{equation*}
\mathrm{d}X_t^{i,N} = b(X_t^{i,N},  \mu_t^{\boldsymbol{X}^{N}} )  \, \mathrm{d}t + \sigma(X_t^{i,N}) \, \mathrm{d}W_t^{i}, \quad X_{0}^{i,N} = \xi^{i}.
\end{equation*}
Here, $W^{i} = (W_t^{i})_{t \in [0,T]}$ and $\xi^{i}$, for $i \in \lbrace 1, \ldots, N \rbrace$, are independent Brownian motions (also independent of $W$) and independent copies of $\xi$, respectively. 
In a next step, one needs to introduce a reasonable time-stepping method to discretise the particles $(X_t^{i,N})_{t \in [0,T]}$ over some finite time horizon $[0,T]$. Numerical schemes for interacting particle systems with H\"{o}lder continuous coefficients (in the state variable) and with coefficients satisfying certain assumptions
 on monotonicity (in the state variable) and Lipschitz continuity (in the measure variable), can be found in \cite{BAO} and \cite{RES,stock,CH} (and the references cited therein), respectively, where a strong convergence analysis is conducted. In \cite{BT,ANK}, a quantitative $L_p$-error analysis in terms of density and cumulative distribution function approximation is presented.
The survey \cite{BO} discusses several examples and numerical schemes for McKean--Vlasov equations involving singular drifts, e.g., a probabilistic interpretation of the Burgers equation, see also \cite{BT}, of the 2D-incompressible Navier--Stokes equation (see e.g., \cite{CHO,OSA}) and turbulent flow models \cite{POP}. Other examples of McKean--Vlasov equations with singular drifts appear in the Keller--Segel equation \cite{JOU}, the Coulomb gas model \cite{CEP}, the Thomson problem \cite{TP}, and the Stefan problem \cite{KR,KRSS}.

Our numerical schemes present an original approximation method which, as of now, is restricted to the specific case of a one-dimensional spatial and one-point discontinuous drift component, but provides, again in this specific framework, a suitable alternative to the methodical
mollification/cut-off approximation methods.

In this article, we first focus on the decomposable case, namely that
\[
b(x,\mu) = b_1(x) + b_2(x,\mu),
\]
where $b_1$ is piecewise Lipschitz continuous with a discontinuity in zero, and $b_2$ satisfies the usual Lipschitz assumptions in both components. This structure allows us to present the main ideas of the analysis to be used later in a more general setting, in particular a transformation of the state variable to remove the discontinuity. In this setting, we prove well-posedness of the McKean--Vlasov equation and the associated particle system. 
This structure includes the important class of McKean--Vlasov equations of the form 
\begin{equation}\label{eq:LIN}
\mathrm{d}X_t = \left( V(X_t) + \int_{\mathbb{R}} \beta(X_t-y) \, \mathcal{L}_{X_t}(\mathrm{d}y) \right) \, \mathrm{d}t + \sigma(X_t) \, \mathrm{d}W_t, \quad X_0 = \xi,
\end{equation}
where $V$ describes an external potential and $\beta$ an interaction kernel; see, e.g., \cite{HREN} and the references cited therein related to mean-field over-damped Langevin equations. These models also embed the class of self-stabilizing diffusions and the McKean--Vlasov model related to the granular media equation.

We then relax the structural assumption on decomposability slightly, but still have to require certain continuity of the measure derivatives at the points of discontinuity, which encompasses the above setting as a special case. The necessity for this condition arises due to the explicit measure or time dependence of the employed transformation. A future research direction concerns a setting where the point of discontinuity is time-dependent, or depends on the distribution of the process $(X_t)_{t \in [0,T]}$, which is relevant to study further practically important models, e.g., from \cite{JOU2}.

Having established the existence of a unique strong solution with bounded moments, we propose two Euler--Maruyama schemes for the particle systems as numerical approximations to the McKean--Vlasov equations. For an Euler--Maruyama scheme applied to the SDE in the transformed state, strong convergence of order 1/2 follows immediately, while for a direct time-discretisation of the particle system without transformation, we are only able to show order 1/9. Numerical tests indicate that this order is in general not sharp. We will discuss the reasons for this gap and possible improvements later.

The main contributions of the present article are as follows. First, we establish the well-posedness of McKean--Vlasov SDEs (with a certain discontinuity) and of their associated particle systems. Techniques from variational calculus on the measure space $\mathcal{P}(\RR)$ equipped with the Wasserstein distance $\mathcal{W}_2$ will be essential in the proofs, due to the possible measure dependence of 
the transformation applied to the processes as described above. The second central contribution of the present paper is the development of numerical schemes for approximating such McKean--Vlasov SDEs and their associated particle systems. Here, a non-standard strong convergence analysis based on occupation time estimates of the discretised processes in a neighbourhood of the discontinuity will be presented. 

The remainder of the paper is organised as follows: In Section \ref{sec:Prelim}, we collect all preliminary tools and notions needed throughout the paper. The precise problem description and the main results are presented in Section \ref{sec:Well-Posed}. Then, Section \ref{sec:isolatednumerics} discusses numerical schemes for McKean--Vlasov SDEs with discontinuous drift. We show strong convergence of certain orders with respect to the number of particles and time-steps, respectively. In Section \ref{SEC:NUM}, we apply our numerical scheme to a model problem arising in neuroscience \cite{FPZ} and to a slight modification of a mean-field game in systemic risk \cite{XGUO,CARM}.
 
\section{Preliminaries}\label{sec:Prelim}

In the sequel, we will introduce several concepts and notions, which will be needed throughout this article. In addition, we will give a brief introduction to the so-called Lions derivative (abbreviated by $L$-derivative), which allows us to define a derivative with respect to measures of the space $\mathcal{P}_2(\RR)$ (see below for a precise definition). Also, we recall the transformation used to cope with drifts having discontinuities in a given finite number of points and first developed in \cite{sz2016a}. We give a summary of important properties of this mapping. Note that generic constants used in this article are denoted by $C>0$. They are independent of the number of particles and number of time-steps, and might change their values from line to line.

\subsection{Notions and notation}
We start with introducing some notions and fixing the notation.
\begin{itemize}
\item Throughout this article, $(\Omega,\mathcal{F},(\mathcal{F}_t)_{t \in [0,T]},\mathbb{P})$ will denote a filtered probability space, where $(\mathcal{F}_t)_{t \in [0,T]}$ is the natural filtration of $W$ augmented with an independent $\sigma$-algebra $\mathcal{F}_0$ and $(\Omega,\mathcal{F},\mathbb{P})$ is assumed to be atomless. 
\item $(\mathbb{R}^d,\left \langle \cdot,\cdot \right \rangle, |\cdot|)$ represents the $d$-dimensional ($d \geq 1$) Euclidean space. As a matrix-norm, we will use $\| A \| := \sup_{v \in \RR^{d}}|Av|$, for any $A \in \RR^{d \times d}$.
\item We use $\mathcal{P}(\mathbb{R})$ to denote the family of all probability
measures on $(\mathbb{R},\mathcal{B}(\mathbb{R}))$, where $\mathcal{B}(\mathbb{R})$ denotes the Borel $\sigma$-field over $\mathbb{R}$ and define the subset of probability measures with finite second moment by
\begin{equation*}
\mathcal{P}_2(\mathbb{R}):= \Big \{ \mu\in \mathcal
{P}(\mathbb{R}) : \ \int_{\mathbb{R}} |x|^2 \mu(\mathrm{d} x)<\infty \Big \}.
\end{equation*} 
\item We recall the definition of the standard Wasserstein distance with quadratic cost: For any $\mu, \nu \in \mathcal{P}_2(\mathbb{R})$, we define
\begin{equation*}
\mathcal{W}_2(\mu, \nu) := \left(\inf_{\pi \in \Pi(\mu,\nu)} \int_{\mathbb{R} \times \mathbb{R}} |x-y |^2 \pi(\mathrm{d}x,\mathrm{d}y) \right)^{1/2},
\end{equation*}
where $\Pi(\mu,\nu)$ denotes the set of all couplings between $\mu$ and $\nu$.
\item For a given $p \geq 2$, $L_p^{0}(\mathbb{R})$ refers to the space of real-valued, $\mathcal{F}_0$-measurable random variables $X$ satisfying $\mathbb{E}[|X|^p] < \infty$ and for a terminal time $T>0$, $\mathcal{S}^p([0,T])$ refers to the space of real-valued continuous, $\mathbb{F}$-adapted processes, defined on the interval $[0,T]$, with finite $p$-th moments, i.e., processes $(X_t)_{t \in [0,T]}$ satisfying $\mathbb{E} \left[ \sup_{t \in [0,T]}|X_t|^p \right] < \infty$.
\end{itemize}
\noindent
We briefly introduce the $L$-derivative of a functional $f: \mathcal{P}_2(\mathbb{R}) \to \mathbb{R}$, as it will appear in the proofs presented in the main section. For further information on this concept, we refer to \cite{PLI} or \cite{BLPR}. Here, we follow the exposition of \cite{CD}. We will associate to the function $f$ a lifted function $\tilde{f}$, defined by $\tilde{f}(X)=f(\mathcal{L}_X)$, where $\mathcal{L}_X$ is the law of $X$, for $X \in L_2(\Omega, \mathcal{F},\mathbb{P};\mathbb{R})$.

This will allow us to introduce $L$-differentiability as Fr\'{e}chet derivative on the lifted space. In particular, a function $f$ on $\mathcal{P}_2(\mathbb{R})$ is said to be $L$-differentiable at $\mu_0 \in \mathcal{P}_2(\mathbb{R})$ if there exists a random variable $X_0 \in L_2(\Omega, \mathcal{F},\mathbb{P};\mathbb{R})$ with law $\mu_0$, such that the lifted function $\tilde{f}$ is Fr\'{e}chet differentiable at $X_0$. 

Now, the Riesz representation theorem implies that there is a ($\mathbb{P}$-a.s.) unique $\Phi \in L_2(\Omega, \mathcal{F},\mathbb{P};\mathbb{R})$ with
\begin{equation*}
\tilde{f}(X) = \tilde{f}(X_0) + \langle \Phi,  X-X_0 \rangle_{L_2}  + o(\| X-X_0\|_{L_2}), \text{ as } \| X-X_0\|_{L_2} \to 0,
\end{equation*}
with the standard inner product and norm on $L_2(\Omega, \mathcal{F},\mathbb{P};\mathbb{R})$. If $f$ is $L$-differentiable for all $\mu_0 \in \mathcal{P}_2(\mathbb{R})$, then we say that $f$ is $L$-differentiable.
 
It is known (see, e.g., \cite[Proposition 5.25]{CD}) that there exists a Borel measurable function $\chi: \mathbb{R} \to \mathbb{R}$, such that $\Phi =  \chi(X_0)$ almost surely, and hence
\begin{equation*}
f(\mathcal{L}_X) = f(\mathcal{L}_{X_0}) + \mathbb{E}\left\langle \chi(X_0), X -X_0 \right \rangle +o(\| X-X_0\|_{L_2}).
\end{equation*}   
Note that $\chi$ only depends on the law of $X_0$, but not on $X_0$ itself.   
We define $\partial_{\mu}f(\mathcal{L}_{X_0})(y):=\chi(y)$, $y \in \mathbb{R}$, as the $L$-derivative of $f$ at $\mu_0$. If, in addition, for a fixed $y \in \mathbb{R}$, there is a version of the mapping $\mathcal{P}_2(\mathbb{R}) \ni \mu \mapsto \partial_{\mu}f(\mu)(y)$ which is continuously $L$-differentiable, then the $L$-derivative of $\partial_{\mu}f(\cdot)(y): \mathcal{P}_2(\mathbb{R}) \to \mathbb{R}$, is defined as
\begin{equation*}
\partial^2_{\mu} f(\mu)(y,y'):= \partial_{\mu}(\partial_{\mu}f)(\cdot)(y)(\mu,y'),
\end{equation*}
for $(\mu,y,y') \in \mathcal{P}_2(\mathbb{R}) \times \mathbb{R} \times \mathbb{R}$. 

We require a definition describing regularity properties of a function $f: \mathcal{P}_2(\mathbb{R}) \to \mathbb{R}$ in terms of the measure derivative
(see \cite{CD,CCD}).  
\begin{definition}  
Let $f: \mathcal{P}_2(\mathbb{R}) \to \mathbb{R}$ be a given functional.
\begin{itemize}
\item We say that $f$ is an element of the class $\mathcal{C}^{(1,1)}_{b}$, if $f$ is continuously $L$-differentiable, for any $\mu$, there is a continuous version of the mapping $\mathbb{R} \ni y \mapsto \partial_{\mu} f(\mu)(y)$ and the derivatives
\begin{align*}
\partial_{\mu} f(\mu)(y), \quad \partial_{y} \lbrace \partial_{\mu} f(\mu)(\cdot) \rbrace (y),
\end{align*}
exist, are bounded and jointly continuous in the variables $(\mu,y)$ such that $y \in \rm{Supp}(\mu)$.
\item We say that $f$ is an element of the class $\mathcal{C}^{(2,1)}_{b}$, if it is an element of $\mathcal{C}^{(1,1)}_{b}$ and in addition the second order Lions derivative $\partial^2_{\mu} f(\mu)(y,y')$ exists, is bounded and is again jointly continuous in the corresponding variables. Also, the joint continuity of all derivatives is here required globally, i.e., for all $(\mu,y,y')$. 
\end{itemize}
\end{definition}
\noindent
We give the following additional remark, which links the $L$-derivative of functions of empirical measures to the standard partial derivatives of its empirical projections. For a functional $f: \mathcal{P}_2(\RR) \to \RR$, we associate with it the finite dimensional projection $f^N: \RR^N \to \RR$ defined as
\begin{equation*}
f^{N}(\boldsymbol{x}^{N}):=f\left(\frac{1}{N} \sum_{j=1}^N \delta_{x_j} \right),
\end{equation*}  
for $\boldsymbol{x}^{N}:=  (x_1, \ldots, x_N)$.
If $f \in \mathcal{C}^{(2,1)}_{b}$, then $f^{N}$ is twice differentiable (in a classical sense) and 
\begin{align*}
& \partial_{x_i} f^{N}(\boldsymbol{x}^{N}) = \frac{1}{N} \partial_{\mu}f\left(\frac{1}{N} \sum_{j=1}^N \delta_{x_j} \right)(x_i), \\
&  \partial_{x_i} \partial_{x_k} f^{N}(\boldsymbol{x}^{N}) = \frac{1}{N} \partial_y  \partial_{\mu}f\left(\frac{1}{N} \sum_{j=1}^N \delta_{x_j} \right)(x_i) \delta_{i,k} +  \frac{1}{N^2} \partial^2_{\mu}f\left(\frac{1}{N} \sum_{j=1}^N \delta_{x_j} \right)(x_i,x_k),
\end{align*}
where $\delta_{i,k}$ is the Kronecker delta, see, e.g., \cite[Proposition 5.35]{CD}.
\subsection{Properties of the transformation G}\label{Sec:transformProp}
In \cite{sz2016a}, the authors consider one-dimensional SDEs of the form 
\[
\mathrm{d}X_t=b(X_t) \, \mathrm{d}t+\sigma(X_t) \, \mathrm{d}W_t, \quad X_0 = x \in \RR,
\] 
with a piecewise Lipschitz continuous drift coefficient $b$ that is 
discontinuous in $K \in \NN$ points 
$\eta_1, \ldots, \eta_K$, and a Lipschitz diffusion
coefficient $\sigma$ that does not vanish in any $\eta_k$.
A mapping $G:\RR \to \RR$ is defined to transform the SDE into one for $Z=G(X)$  
with globally
Lipschitz continuous coefficients. For simplicity, we restrict the discussion to $K=1$
with $\eta_1=0$. We define the mapping $G$ by
\begin{align*}
G(x) := x + \alpha x |x| \phi\left( \frac{x}{c} \right),
\end{align*} 
where 
\begin{align*}
\phi(x):=
\begin{cases}
(1-x^2)^3, & |x|\le 1, \\
0, & |x|>1\, 
\end{cases}  
\quad \alpha: =\frac{b(0^{-})- b(0^{+})}{2\sigma^2(0)},
\end{align*}
and $c$ is a constant satisfying $0 < c < 1/|\alpha|$. The choice of $\alpha$ yields a Lipschitz continuous drift coefficient for the SDE of $Z=G(X)$, in particular, it removes the discontinuity in $0$ from the drift. The restriction on $c$ guarantees that $G$ possesses a global inverse. 

It is known from \cite{sz2016b} that $G$ satisfies the following properties: 
\begin{itemize}
\item $G$ is $\mathcal{C}^{1}(\RR,\RR)$ with $0 < \inf_{x \in \RR} G'(x) \leq \sup_{x \in \RR} G'(x) < \infty$. Therefore, $G$ is Lipschitz continuous and has an inverse $G^{-1}: \RR \to \RR$ that is Lipschitz continuous as well. 
\item The derivative $G'$ is Lipschitz continuous (i.e., also absolutely continuous). In addition, $G'$ has a bounded Lebesgue density $G'':\RR \to \RR$, which is Lipschitz continuous on each of the subintervals $(-\infty,0)$ and $(0,\infty)$. Also, It\^{o}'s formula can still be applied to $G$ and $G^{-1}$.
\end{itemize}

\section{Existence and uniqueness results}\label{sec:Well-Posed}

The following subsections are devoted to proving well-posedness results for certain classes of one-dimensional McKean--Vlasov SDEs with a drift having a discontinuity in zero. In a first step, we study a simple class where the resulting transformation will not depend on the measure. Here, the transformation techniques developed in \cite{sz2016a} will allow us to prove existence and uniqueness of a strong solution. The second class of McKean--Vlasov SDEs investigated below has the intrinsic difficulty that the required transformation will depend on the measure (i.e., will be time-dependent). Hence, a fixed-point iteration in the measure component will be required and we need to use techniques from variational calculus on the measure space $\mathcal{P}_2(\RR)$, in particular an It\^{o} formula for functionals acting on this space. 

For each of these classes of McKean--Vlasov SDEs, we will additionally study the well-posedness of their associated interacting particle system. Although they can be considered as $N$-dimensional classical SDEs, with $N$ denoting the number of particles, the resulting set of discontinuities of the $N$-dimensional drift cannot be handled by the main results of \cite{sz2016b}.

Future work is needed to extend the methods developed in this article to a multi-dimensional framework. In particular, it seems that the decomposable case can be generalised when discontinuities of the form discussed in \cite{sz2016b} are considered. 

\subsection{McKean--Vlasov SDEs and interacting particle systems with decomposable drift}

For a given terminal time $T >0$ and given $p \geq 2$, we consider a one-dimensional McKean--Vlasov SDE of the form 
\begin{equation}\label{eq:Model1}
\mathrm{d} X_t = b(X_t, \mathcal{L}_{X_t}) \, \mathrm{d}t  + \sigma(X_t) \, \mathrm{d}W_t, \quad \ X_0= \xi \in L_p^{0}(\mathbb{R}),
\end{equation}
where $b:\RR \times \mathcal{P}_2(\RR) \to \RR$ and $\sigma: \RR \to \RR$ are measurable functions.

In the following, we state the model assumptions which will specify the set-up for this subsection: 
\begin{Assumption}\label{Assum:A}
\begin{enumerate}[(1)]
    \item \label{Assum:A1}
    We have $\sigma(0) \neq 0$ and there exists a constant $L >0$ such that
    \begin{equation*}
       | \sigma(x)  - \sigma(x') | \leq L |x-x'| \quad \forall  x, x' \in \mathbb{R}.
    \end{equation*}      
    \item \label{Assum:A2} 
    The drift is decomposable
		in the following sense:
    \begin{equation*}
       b(x, \mu) = b_1(x) + b_2(x,\mu) \quad \forall  x \in \mathbb{R}, \ \forall \mu \in  \mathcal{P}_2(\mathbb{R}),
    \end{equation*}
	where $b_1: \mathbb{R} \to \mathbb{R}$ is Lipschitz continuous 
	 on the subintervals $(-\infty,0)$ and $(0,\infty)$ and there exists a constant $L_1>0$ such that
      \begin{equation*}
       | b_2(x,\mu)  - b_2(x',\nu) | \leq L_1 \left( |x-x'| + \mathcal{W}_2(\mu,\nu) \right) \quad \forall  x, x' \in \mathbb{R}, \ \forall \mu, \nu \in  \mathcal{P}_2(\mathbb{R}).
    \end{equation*}  
\end{enumerate} 
\end{Assumption}
\noindent
We now state the main results of this section:
\begin{proposition}\label{Prop1Ex}
Let Assumption (H.\ref{Assum:A}) be satisfied, let $\xi \in L_p^{0}(\mathbb{R})$ for a given $p \geq 2$ and assume $c < 1/|\alpha|$. Then, the McKean--Vlasov SDE defined in (\ref{eq:Model1}) has a unique strong solution in $\mathcal{S}^{p}([0,T])$.
\end{proposition}
\begin{proof}
Transforming the McKean--Vlasov SDE \eqref{eq:Model1}, employing the transformation $G: \mathbb{R} \to \mathbb{R}$ defined in Section \ref{Sec:transformProp}, with 
\begin{equation*}
\alpha =\frac{b(0^{-}, \mu)- b(0^{+},\mu)}{2\sigma^2(0)} = \frac{b_1(0^{-}) - b_1(0^{+})}{2 \sigma^2(0)},  
\end{equation*}
in order to eliminate the discontinuity in zero, yields a McKean--Vlasov SDE with globally Lipschitz continuous coefficients. This can be shown in a similar manner to \cite[Theorem 2.5]{sz2016a}). Moreover, $G$ has a global inverse due to the choice $c < 1/|\alpha|$ (see \cite[Lemma 2.2]{sz2016b}), and It\^{o}'s formula can be applied to $G^{-1}$, which allows to deduce the claim.

\end{proof} 

The interacting particles 
$(X^{i,N}_t)_{t \in [0,T]}$, $i \in \lbrace 1, \ldots, N \rbrace$, associated with (\ref{eq:Model1}) satisfy
\begin{equation}\label{eq:IPSystem}
\mathrm{d}X_t^{i,N} = b_1(X_t^{i,N}) \, \mathrm{d}t + b_2(X_t^{i,N}, \mu_t^{\boldsymbol{X}^{N}}) \, \mathrm{d}t + \sigma(X_t^{i,N}) \, \mathrm{d}W_t^{i},
\end{equation} 
where $(\xi^{i},W^{i})$, for $i \in \lbrace 1, \ldots, N \rbrace$, are independent copies of $(\xi,W)$. 
\begin{proposition}
Let Assumption (H.\ref{Assum:A}) be satisfied, let $\xi \in L_p^{0}(\mathbb{R})$ for a given $p \geq 2$ and assume $c < 1/|\alpha|$. Then, the interacting particle system defined in (\ref{eq:IPSystem}) 
has a unique strong solution in $\mathcal{S}^{p}([0,T])$.
\end{proposition}
\begin{proof}
In contrast to the set-up in \cite{sz2016b}, the set of
discontinuities, denoted by $\Theta$, is not a differentiable manifold, but has the form
\begin{equation*}
\Theta=\{(x_1,\dots,x_N)^{\top} \in \RR^N\colon \exists j\in \{1,\dots,N\}\colon x_j=0\}.
\end{equation*}
However, we may define $\boldsymbol{G}_N: \RR^{N} \to \RR^{N}$ by 
\begin{equation*}
\boldsymbol{G}_N(\boldsymbol{x}^N) := (G(x_1), \ldots, G(x_N))^{\top},
\end{equation*}
where $G$ is as in Proposition \ref{Prop1Ex} and $\boldsymbol{x}^N=(x_1, \ldots, x_N)^{\top}$, which allows us to transform the particle system $(\boldsymbol{X}_t^{N})_{t \in [0,T]} = (X^{1,N}_t, \ldots, X^{N,N}_t)^{\top}_{t \in [0,T]}$ into a new particle system with globally Lipschitz continuous coefficients. Now, $\boldsymbol{G}_N$ has a global inverse, as the mapping $G$ has a global inverse, due to the choice of $c$ (see Section \ref{Sec:transformProp}). Therefore, applying It\^{o}'s formula to the inverse allows to deduce the claim.

\end{proof}  

\subsection{McKean--Vlasov SDE with non-decomposable drift}
Here, we consider again a one-dimensional McKean--Vlasov SDE of the form \eqref{eq:Model1},
\begin{equation}\label{eq:Model2}
\mathrm{d} X_t = b(X_t, \mathcal{L}_{X_t}) \, \mathrm{d}t  + \sigma(X_t) \, \mathrm{d}W_t, \quad \ X_0= \xi \in L_p^{0}(\mathbb{R}).
\end{equation}
However, in contrast to the above setting, we will not assume that $b$ can be decomposed in two parts as in Assumption (H.\ref{Assum:A}(\ref{Assum:A2})) from the previous section, and therefore the transformation will also depend on the measure. To be precise, for any $(x,\mu) \in \RR \times \mathcal{P}_2(\RR)$, we define
\begin{align}\label{eq:TransMeas}
G(x,\mu) := x + \alpha(\mu) x |x| \phi\left( \frac{x}{c} \right),
\end{align} 
where 
\begin{align}\label{eq:TransMeas2}
&\phi(x):=
\begin{cases}
(1-x^2)^3, & |x|\le 1, \\
0, & |x|>1\, 
\end{cases} 
\qquad \alpha(\mu): =\frac{b(0^{-},\mu)- b(0^{+},\mu)}{2\sigma^2(0)}, 
\end{align}
and $c>0$ is a constant small enough to guarantee the invertibility of $G$. When we speak of an `inverse' of $G(x,\mu)$, we mean `inverse with respect to $x$', i.e., the inverse is a function $G^{-1}\colon \RR\times \mathcal{P}_2(\RR) \to \RR$ which satisfies $G^{-1} \big(G(x,\mu),  \mu \big)=x$ for all $(x,\mu) \in \RR\times \mathcal{P}_2(\RR)$ and $G\big(G^{-1}(z,\mu), \mu \big)=z$ for all $(z,\mu)\in \RR\times \mathcal{P}_2(\RR)$. For a given flow of measures $(\mu_t)_{t \in [0,T]} \in \mathcal{C}([0,T],\mathcal{P}_2(\RR))$, $G$ may also be viewed as a mapping $G \colon \RR \times [0,T] \to \RR$.

In the following, we state the model assumptions which will specify the set-up for this subsection: 
\begin{Assumption}\label{Assum:AA}
Assumption (H.\ref{Assum:A}(\ref{Assum:A1})) is satisfied and we require:
\begin{enumerate}[(1)]
    \item \label{Assum:AA2} 
    There exist constants $L, L_1>0$ such that  
      \begin{align*}
       & \sup_{x \neq 0} \frac{|b(x,\mu)|}{1+|x|} \leq L,  \quad | b(x,\mu)  - b(x,\nu) | \leq L_1 \mathcal{W}_2(\mu,\nu) \quad \forall x \neq 0 \in \mathbb{R}, \ \forall \mu, \nu \in  \mathcal{P}_2(\mathbb{R}). 
    \end{align*}
     Additionally, for all $\mu \in \mathcal{P}_2(\RR)$, $\RR \ni x \mapsto b(x, \mu)$
     is piecewise Lipschitz continuous on the subintervals $(-\infty,0)$ and $(0,\infty)$, uniformly with respect to $\mu$.
     \item \label{Assum:AA3}   
     $\alpha \in \mathcal{C}^{(1,1)}_{b}$, and the mapping $\mathcal{P}_2(\RR) \times \RR \ni (\mu, y) \mapsto \partial_y \partial_{\mu} \alpha(\mu)(y)$ is Lipschitz continuous, that is, there exists a constant $L_2>0$ such that 
     \begin{align*}
      & | \partial_y \partial_{\mu} \alpha(\mu)(y) -  \partial_y \partial_{\mu} \alpha(\nu)(y')| \leq L_2 \left(|y-y'| + \mathcal{W}_2(\mu,\nu) \right) \quad \forall y,y' \in \mathbb{R}, \ \forall \mu, \nu \in  \mathcal{P}_2(\mathbb{R}).
     \end{align*}      
      \item \label{Assum:AA4}
      For any $\mu \in \mathcal{P}_2(\RR)$, the mapping $\RR \ni y \mapsto \partial_{\mu} \alpha(\mu)(y)$
      vanishes in zero, i.e., 
      \begin{align}\label{eq:zeroinzero}
       & \partial_{\mu} \alpha(\mu)(0) = \partial_{\mu} b(0^{-}, \mu)(0)- \partial_{\mu}b(0^{+},\mu)(0) = 0. 
      \end{align}
      \item \label{Assum:AA5}
      The mapping $\mathcal{P}_2(\RR) \times \RR \ni (\mu,y) \mapsto \partial_{\mu}\alpha(\mu)(y) b(y,\mu)$ is Lipschitz continuous.     
\end{enumerate} 
\end{Assumption}
\noindent
\begin{remark}
The requirement in (H.\ref{Assum:AA}(\ref{Assum:AA3})) that $\alpha \in \mathcal{C}^{(1,1)}_{b}$ is needed to apply an It\^{o} formula for $\alpha$ (see \cite[Proposition 5.102]{CD}).
\end{remark}
\begin{remark}
Note that (H.\ref{Assum:AA}(\ref{Assum:AA5})) could also be replaced by the following alternative set of assumptions:
On each of the two subintervals $(-\infty,0)$ and $(0,\infty)$, $b(\cdot,\mu)$ is a $\mathcal{C}^{1}(\RR,\RR)$ function with bounded derivative and, additionally, for any $\mu \in \mathcal{P}_2(\RR)$, the mapping $\RR \ni y \mapsto \partial_y\partial_{\mu} \alpha(\mu)(y)$ vanishes in zero, i.e.,
\begin{equation*}
\partial_y \partial_{\mu} \alpha(\mu)(0) = \partial_y\partial_{\mu} b(0^{-}, \mu)(0)- \partial_y\partial_{\mu}b(0^{+},\mu)(0) = 0.
\end{equation*}
\end{remark}
\noindent
The following proposition shows the Lipschitz continuity of the mapping $\RR \times \mathcal{P}_2(\RR) \ni (x,\mu) \to G^{-1}(x,\mu)$.
\begin{proposition}\label{Prop}
Let the function $G$ be defined as in \eqref{eq:TransMeas} with $c < 1/ \sup_{\mu \in \mathcal{P}_2(\RR)}|\alpha(\mu)|$ and
\eqref{eq:TransMeas2} and let Assumption (H.\ref{Assum:AA}(\ref{Assum:AA2})) be
satisfied. Then, there exists a constant $L(c)$ satisfying $L(c) \to 0$ as $c \to 0$, such that for any $x,y \in \RR$ and $\mu, \nu \in \mathcal{P}_2(\RR)$ 
\begin{equation*}
|G^{-1}(x,\mu) - G^{-1}(y,\nu)| \leq 2|x-y| + L(c)\mathcal{W}_2(\mu,\nu).
\end{equation*}
\end{proposition}
\begin{proof}
First, we note that by differentiating \eqref{eq:TransMeas}, we get for $x \in [-c,c]$
\begin{align*}
\partial_x G(x,\mu) 
&= 1 - \frac{6\alpha(\mu)}{c^2} |x| x^2 \left(1-(x/c)^2 \right)^2 + 2\alpha(\mu)|x|(1-(x/c)^2)^3\\
&= 1 - 2c \alpha(\mu) \frac{|x|}{c}\big(4x^2/c^2 - 1\big)\left(1-(x/c)^2 \right)^2\,.
\end{align*}
It is easy to verify that 
\[
\sup_{x\in [0,c]}\Big|\frac{|x|}{c}\big(4x^2/c^2 - 1\big)\left(1-(x/c)^2 \right)^2\Big|
=\sup_{z\in [0,1]}\Big||z|\big(4z^2    - 1\big)\left(1-z^2 \right)^2\Big|<\frac{1}{4}\,,
\]
which implies that for all $x\in \RR$ and $\mu\in \mathcal{P}_2(\RR)$
\[
\big|\partial_x G(x,\mu)-1\big|
<c \frac{|\sup_{\mu \in \mathcal{P}_2(\RR)}\alpha(\mu)|}{2}\,.
\]
In particular, if $c < 1/\sup_{\mu \in \mathcal{P}_2(\RR)} |\alpha(\mu)|$, then 
for all $x,y\in \RR$ and $\mu\in \mathcal{P}_2(\RR)$, we have
\begin{align*}
\partial_x G(x,\mu)>\frac{1}{2}, \quad
\big|G^{-1}(x,\mu)-G^{-1}(y,\mu)\big|<2|x-y|\,.
\end{align*}
That
$ \mu \mapsto \partial_xG(x,\mu)$ is Lipschitz continuous
is a consequence of (H.\ref{Assum:AA}(\ref{Assum:AA2})).  
It is easy to show that the mapping $ x \mapsto \partial_xG(x,\mu)$ is also Lipschitz continuous. Denote the Lipschitz
constant of $\partial_xG$ with respect to the first and second argument by 
$L_x$ and $L_{\mu}$, respectively.
Writing 
\begin{equation*}
G^{-1}(x,\mu) = \int_{0}^{x} \frac{1}{\partial_xG(G^{-1}(y,\mu),\mu)} \, \mathrm{d}y, 
\end{equation*}
we obtain
\begin{align*}
& |G^{-1}(x,\mu) - G^{-1}(x,\nu)|  \\
& \leq  \int_{0}^{x} \left| \frac{1}{\partial_xG(G^{-1}(y,\mu),\mu)} - \frac{1}{\partial_xG(G^{-1}(y,\nu),\mu)} \right| \, \mathrm{d}y  + \int_{0}^{x} \left| \frac{1}{\partial_xG(G^{-1}(y,\nu),\mu)} - \frac{1}{\partial_xG(G^{-1}(y,\nu),\nu)} \right| \, \mathrm{d}y \\
& \leq  \int_{0}^{x} \left| \frac{\partial_xG(G^{-1}(y,\nu),\mu)-\partial_xG(G^{-1}(y,\mu),\mu)}{\partial_xG(G^{-1}(y,\mu),\mu)\partial_xG(G^{-1}(y,\nu),\mu)} \right| \, \mathrm{d}y  
+ \int_{0}^{x} \left| \frac{\partial_xG(G^{-1}(y,\nu),\nu)-\partial_xG(G^{-1}(y,\nu),\mu)}{\partial_xG(G^{-1}(y,\nu),\mu)\partial_xG(G^{-1}(y,\nu),\nu)} \right| \, \mathrm{d}y \\
& \leq  \int_{0}^{x} 4L_x\left| G^{-1}(y,\nu)-G^{-1}(y,\mu) \right| \, \mathrm{d}y  
+ \int_{0}^{x} 4L_{\mu} \mathcal{W}_2(\mu,\nu) \, \mathrm{d}y \\
& \leq  4L_x\int_{0}^{x} \left| G^{-1}(y,\nu)-G^{-1}(y,\mu) \right| \, \mathrm{d}y  
+  4L_{\mu}x  \mathcal{W}_2(\mu,\nu)  \,.
\end{align*}

For $|x|< c$, we have 
\begin{align*}
|G^{-1}(x,\mu) - G^{-1}(x,\nu)|  
& \leq  4L_x\int_{0}^{x} \left| G^{-1}(y,\nu)-G^{-1}(y,\mu) \right| \, \mathrm{d}y  +  4L_{\mu}x  \mathcal{W}_2(\mu,\nu)  \\
& \leq  4L_x\int_{0}^{x} \left| G^{-1}(y,\nu)-G^{-1}(y,\mu) \right| \, \mathrm{d}y  +  4L_{\mu}c  \mathcal{W}_2(\mu,\nu)  \,,
\end{align*}
and hence Gronwall's inequality implies
\[
|G^{-1}(x,\mu) - G^{-1}(x,\nu)|
\le 4L_{\mu}c  \mathcal{W}_2(\mu,\nu)e^{4L_xx}
\le 4L_{\mu}ce^{4L_xc}  \mathcal{W}_2(\mu,\nu)\,.
\]
For $|x| \geq c$, $|G^{-1}(x,\mu) - G^{-1}(x,\nu)|=0\le 4L_{\mu}ce^{4L_xc}  \mathcal{W}_2(\mu,\nu)$ by the definition of $G$. 
We finally obtain, for all $x,y\in\RR$ and
all $\mu,\nu\in \mathcal{P}_2(\RR)$, with $L(c):=4L_{\mu}ce^{4L_xc}$,
\begin{align*}
|G^{-1}(x,\mu) - G^{-1}(y,\nu)|
&\le|G^{-1}(x,\mu) - G^{-1}(x,\nu)|+|G^{-1}(x,\nu) - G^{-1}(y,\nu)|\\
&\le L(c)  \mathcal{W}_2(\mu,\nu) + 2|x-y|
\le \max(L(c),2)  (\mathcal{W}_2(\mu,\nu) + |x-y|)\,.
\end{align*}
\end{proof}
\begin{remark}\label{REM:LIP}
In what follows, we will assume that $c < 1/\sup_{\mu \in \mathcal{P}_2(\RR)} |\alpha(\mu)|$ and is small enough such that the Lipschitz constant of the mapping $\mathcal{P}_2(\RR) \ni \mu \mapsto G^{-1}(x,\mu)$ (i.e., the constant $L(c)$ a few lines above) is less than a half. The reason for this requirement will become clearer in the proof of Theorem \ref{TH:MAIN}.
\end{remark}
Similar to the previous section, we aim to recover a unique strong solution of \eqref{eq:Model2} by setting $X_t = G^{-1}(Z_t^{\mu},\mu_t)$, where $\mu_t = \mathcal{L}_{X_t}$ for $t \in [0,T]$, and $(Z_t^{\mu})_{t \in [0,T]}$ is the process obtained by applying the transformation $G$ to $X$. 
Even though $G$ is not twice continuously differentiable in the state variable, It\^{o}'s formula is still applicable due to the special form of the discontinuity
(see \cite[Theorem 2.1]{EL} and the comments after the proof of this theorem). 
Now, observe that 
\begin{align*}
\mathrm{d}G(X_t,\mu_t) & = \left( \partial_t G(X_t,\mu_t) +  b(X_t,\mu_t) + \alpha({\mu_t}) \bar{\phi}'(X_t) b(X_t,\mu_t) + \frac{1}{2} \alpha({\mu_t})  \bar{\phi}''(X_t) \sigma^2(X_t) \right) \, \mathrm{d}t  \\
& \qquad + (\sigma(X_t) + \alpha({\mu_t}) \bar{\phi}'(X_t)\sigma(X_t)) \, \mathrm{d}W_t.
\end{align*}
It\^{o}'s formula along a flow of measures $(\mu_t)_{t \in [0,T]} \in \mathcal{C}([0,T],\mathcal{P}_2(\RR))$ (see, e.g., \cite[Proposition 5.102]{CD}) implies
\begin{align*}
\partial_t G(x,\mu_t) &= \int_{\RR} \left( b(y,\mu_t) \partial_{\mu}G(x,\mu_t)(y) + \frac{\sigma^{2}(y)}{2} \partial_y \partial_{\mu}G(x,\mu_t)(y) \right) \, \mu_t(\mathrm{d}y)=:\mathcal{L}_{\mu_t}(G(x,\cdot))(\mu_t),
\end{align*} 
where we recall that $\partial_y \partial_{\mu}G(x,\mu_t)(y)$ denotes the derivative of the mapping $\RR \ni y \mapsto \partial_{\mu}G(x,\mu_t)(y)$ and
\begin{align*}
& \partial_{\mu}G(x,\mu_t)(y) = \partial_{\mu}\alpha(\mu_t)(y)|x|x\phi(x/c), \quad \partial_y \partial_{\mu}G(x,\mu_t)(y) = \partial_y\partial_{\mu}\alpha(\mu_t)(y)|x|x\phi(x/c).
\end{align*}
Hence, we define
\begin{align}\label{eq:Transf}
& \mathrm{d}Z_t^{\mu} :=\tilde{b}(Z^{\mu}_t,\mu_t) \, \mathrm{d}t + \tilde{\sigma}(Z^{\mu}_t) \, \mathrm{d}W_t, \quad  Z_0^{\mu} = G(\xi,\delta_{\xi}),
\end{align} 
where 
\begin{align}\label{eq:TransfCoeff}
\tilde{b}(z,\mu) &:= \mathcal{L}_{\mu}(G(G^{-1}(z,\mu),\cdot)(\mu)  +  b(G^{-1}(z,\mu),\mu) + \alpha({\mu}) \bar{\phi}'(G^{-1}(z,\mu)) b(G^{-1}(z,\mu),\mu) \notag \\
& \quad  + \frac{1}{2} \alpha(\mu)  \bar{\phi}''(G^{-1}(z,\mu)) \sigma^2(G^{-1}(z,\mu)),  \notag \\
\tilde{\sigma}(z,\mu) &:= \sigma(G^{-1}(z,\mu)) + \alpha({\mu}) \bar{\phi}'(G^{-1}(z,\mu))\sigma(G^{-1}(z,\mu)).
\end{align}
In the following, we will show that the decoupled SDE \eqref{eq:Transf} where the flow $(\mu_t)_{t \in [0,T]} \in \mathcal{C}([0,T],\mathcal{P}_2(\RR))$ is fixed has Lipschitz continuous coefficients. Note that for a such a flow of measures the process $X^{\mu}$ in \eqref{eq:Model2} (interpreted as classical SDE) has bounded moments uniformly in $(\mu_t)_{t \in [0,T]}$, which is a consequence of (H.\ref{Assum:A}(\ref{Assum:A1})) and (H.\ref{Assum:AA}(\ref{Assum:AA2})). In particular, for $\xi \in L_2^{0}(\mathbb{R})$, we have the a-priori estimate $\mathbb{E}\left[\sup_{0 \leq t \leq T} |X^{\mu}_t|^{2} \right] \leq C(1+ \mathbb{E}[|\xi|^{2}])=:\bar{C}$,
where $C>0$ only depends on $T,p$ and the constants appearing in the model assumptions. We will introduce the following subspace of $\mathcal{C}([0,T],\mathcal{P}_2(\RR))$: We define $\mathcal{P}^{b} := \lbrace \mu \in \mathcal{C}([0,T],\mathcal{P}_2(\RR)): \ \sup_{t \in [0,T]} \int_{\RR} x^2 \, \mu_t(\mathrm{d}x) \leq \bar{C} \rbrace$, where $\bar{C}$ is defined as above and complete this space with the metric $\sup_{t \in [0,T]} \mathcal{W}_2(\mu_t,\nu_t)$, for $(\mu_t)_{t \in [0,T]}, (\nu_t)_{t \in [0,T]} \in \mathcal{P}^{b}$.
\begin{lemma}\label{lem:Lip}
Let Assumption (H.\ref{Assum:AA}) be satisfied and assume $c < 1/ \sup_{\mu \in \mathcal{P}_2(\RR)}|\alpha(\mu)|$. Then, $\tilde{b}$ and $\tilde{\sigma}$ given in \eqref{eq:TransfCoeff} are Lipschitz continuous on $\RR \times \mathcal{P}^{b}$.
\end{lemma}
\begin{proof}
Using (H.\ref{Assum:A}(\ref{Assum:A1})), the Lipschitz continuity of $z \mapsto G^{-1}(z,\nu_t)$ and the uniform boundedness of $\alpha$, see (H.\ref{Assum:AA}(\ref{Assum:AA2})), in combination with \cite[Lemma 2.5]{sz2016b}, gives the Lipschitz continuity of $ z \mapsto \tilde{\sigma}(z,\nu_t)$, with a Lipschitz constant independent of $\nu_t$. Similarly, we can deduce that $\nu_t \mapsto \tilde{\sigma}(z,\nu_t)$ is Lipschitz continuous, due to Proposition \ref{Prop} and the Lipschitz continuity of $ \nu_t \mapsto \alpha(\nu_t)$.

The choice of $\alpha$ along with (H.\ref{Assum:AA}(\ref{Assum:AA2})) guarantees that the mapping 
\begin{equation*}
z \mapsto b(G^{-1}(z,\nu_t),\nu_t) + \frac{1}{2} \alpha(\nu_t)  \bar{\phi}''(G^{-1}(z,\nu_t)) \sigma^2(G^{-1}(z,\nu_t)),
\end{equation*}
is Lipschitz continuous. From \cite[Lemma 2.4]{sz2016b}, we can deduce the Lipschitz continuity of 
\begin{equation*}
z \mapsto \alpha(\nu_t) \bar{\phi}'(G^{-1}(z,\nu_t)) b(G^{-1}(z,\nu_t),\nu_t).
\end{equation*}
That 
\begin{align*}
& \nu_t \mapsto b(G^{-1}(z,\nu_t),\nu_t) + \frac{1}{2} \alpha(\nu_t)  \bar{\phi}''(G^{-1}(z,\nu_t)) \sigma^2(G^{-1}(z,\nu_t)), \quad \nu_t \mapsto  \alpha(\nu_t) \bar{\phi}'(G^{-1}(z,\nu_t)) b(G^{-1}(z,\nu_t),\nu_t),
\end{align*}
are Lipschitz continuous is a consequence of (H.\ref{Assum:A}(\ref{Assum:A1})) and (H.\ref{Assum:AA}(\ref{Assum:AA2})), Proposition \ref{Prop} and the fact that $G^{-1}(z,\mu^{(1)})$ and $G^{-1}(z,\mu^{(2)})$, for $\mu^{(1)}, \mu^{(2)} \in \mathcal{P}_2(\RR)$, have the same sign. Also note that $\bar{\phi}'$ and $\bar{\phi}''$ are (piecewise) Lipschitz continuous and bounded. It remains to analyse the Lipschitz continuity of 
\begin{align}\label{eq:MEAS}
(z,\nu_t) \mapsto \int_{\RR} \left( b(y,\nu_t) \partial_{\mu}G(G^{-1}(z,\nu_t),\nu_t)(y) + \frac{\sigma^{2}(y)}{2} \partial_y \partial_{\mu}G(G^{-1}(z,\nu_t),\nu_t)(y) \right) \, \nu_t(\mathrm{d}y). 
\end{align}
Assumptions (H.\ref{Assum:AA}(\ref{Assum:AA3})) and (H.\ref{Assum:AA}(\ref{Assum:AA4})) guarantee that above mapping exists and further that the mapping $y \mapsto b(y,\nu_t) \partial_{\mu}G(G^{-1}(z,\nu_t),\nu_t)(y)$ is continuous in zero. We start by analysing the Lipschitz continuity of \eqref{eq:MEAS} with respect to the measure variable.
Consider now an arbitrary coupling $\Pi_t(\cdot,\cdot)$ between $\nu_t(\cdot)$ and $\mu_t(\cdot)$, for $(\mu_t)_{t \in [0,T]}, (\nu_t)_{t \in [0,T]} \in \mathcal{P}^{b}$, and estimate
\begin{align*}
&\int_{\RR^2} \left( b(y,\nu_t) \partial_{\mu}G(G^{-1}(z,\nu_t),\nu_t)(y) - b(x,\mu_t) \partial_{\mu}G(G^{-1}(z,\mu_t),\mu_t)(x) \right) \, \Pi_t(\mathrm{d}y,\mathrm{d}x) \\
& = \int_{\RR^2} \left( b(y,\nu_t) \partial_{\mu}G(G^{-1}(z,\nu_t),\nu_t)(y) - b(x,\mu_t) \partial_{\mu}G(G^{-1}(z,\nu_t),\mu_t)(x) \right) \, \Pi_t(\mathrm{d}y,\mathrm{d}x) \\
& \quad + \int_{\RR^2} \left(b(x,\mu_t) \partial_{\mu}G(G^{-1}(z,\nu_t),\mu_t)(x) - b(x,\mu_t) \partial_{\mu}G(G^{-1}(z,\mu_t),\mu_t)(x) \right) \, \Pi_t(\mathrm{d}y,\mathrm{d}x).
\end{align*}
Note that 
\begin{align*}
& \int_{\RR^2} \left| b(y,\nu_t) \partial_{\mu}G(G^{-1}(z,\nu_t),\nu_t)(y) - b(x,\mu_t) \partial_{\mu}G(G^{-1}(z,\nu_t),\mu_t)(x) \right| \, \Pi_t(\mathrm{d}y,\mathrm{d}x) \\
& =  \int_{\RR^2}  \Big| b(y,\nu_t) \partial_{\mu}\alpha(\nu_t)(y) \bar{\phi}(G^{-1}(z,\nu_t)) - b(x,\mu_t) \partial_{\mu}\alpha(\mu_t)(x)\bar{\phi}(G^{-1}(z,\nu_t)) \Big| \, \Pi_t(\mathrm{d}y,\mathrm{d}x) \\
& \leq C \mathcal{W}_2(\mu_t,\nu_t),
\end{align*}
where we used (H.\ref{Assum:AA}(\ref{Assum:AA5})) and the fact that $\bar{\phi}$ is bounded. Furthermore, from the boundedness of $(x,\nu_t) \mapsto \partial_{\mu}\alpha(\nu_t)(x)$ and (H.\ref{Assum:AA}(\ref{Assum:AA2})), we derive
\begin{align}\label{eq:ML}
&\int_{\RR^2} \left|b(x,\mu_t) \partial_{\mu}G(G^{-1}(z,\nu_t),\mu_t)(x) - b(x,\mu_t) \partial_{\mu}G(G^{-1}(z,\mu_t),\mu_t)(x) \right| \, \Pi_t(\mathrm{d}y,\mathrm{d}x) \notag \\
& \leq \int_{\RR^2} \left|  b(x,\mu_t)\partial_{\mu}\alpha(\mu_t)(x)  \right| \left| \bar{\phi}(G^{-1}(z,\nu_t)) - \bar{\phi}(G^{-1}(z,\mu_t)) \right| \, \Pi_t(\mathrm{d}y,\mathrm{d}x) \notag \\
& \leq C  \int_{\RR^2} (1+|x|) \left | \bar{\phi}(G^{-1}(z,\nu_t)) - \bar{\phi}(G^{-1}(z,\mu_t)) \right| \, \Pi_t(\mathrm{d}y,\mathrm{d}x) \notag \\
& \leq C \mathcal{W}_2(\mu_t,\nu_t).
\end{align}
We remark that in the last inequality, we used the Lipschitz continuity of $ z \mapsto |z| z \phi(z/c)$, Proposition \ref{Prop} and employed that $(\mu_t)_{t \in [0,T]}$ is an element of the space $\mathcal{P}^{b}$. In a similar manner, we can show the Lipschitz continuity of 
\begin{equation*}
z \mapsto \int_{\RR} b(y,\nu_t) \partial_{\mu}G(G^{-1}(z,\nu_t),\nu_t)(y) \, \nu_t(\mathrm{d}y).
\end{equation*}
Analogous statements can be derived for
\begin{align*}
(z,\nu_t) \mapsto \int_{\RR} \frac{\sigma^{2}(y)}{2} \partial_y \partial_{\mu}G(G^{-1}(z,\nu_t),\nu_t)(y) \, \nu_t(\mathrm{d}y),
\end{align*}
taking (H.\ref{Assum:A}(\ref{Assum:A1})) and (H.\ref{Assum:AA}(\ref{Assum:AA3})) into account.
\end{proof}
We are now ready to present the main result and its proof of this section:
\begin{theorem}\label{TH:MAIN}
Let Assumption (H.\ref{Assum:AA}) be satisfied, let $\xi \in L_p^{0}(\mathbb{R})$ for a given $p \geq 2$ and assume that the constant $c$ is sufficiently small (as in Remark \ref{REM:LIP}). Then, the McKean--Vlasov SDE defined in (\ref{eq:Model2}) has a unique strong solution in $\mathcal{S}^{p}([0,T])$.
\end{theorem}
\begin{proof}
First, we remark that for any given flow of measures $(\mu_t)_{t \in [0,T]} \in \mathcal{C}([0,T],\mathcal{P}_2(\RR))$ the SDE defined in \eqref{eq:Transf} has a unique strong solution by Lemma \ref{lem:Lip}. Now, for $(\mu_t)_{t \in [0,T]}, (\nu_t)_{t \in [0,T]} \in  \mathcal{P}^{b}$, we obtain, for any $t \in [0,T]$, using Lemma \ref{lem:Lip}, BDG's inequality and H\"{o}lder's inequality along with Gronwall's inequality
\begin{align*}
\mathbb{E}\left[ |Z^{\mu}_t-Z^{\nu}_t|^2 \right] & \leq  C \left(  \mathbb{E} \left[ \int_{0}^{t} |\tilde{b}(Z^{\mu}_s,\mu_s) - \tilde{b}(Z^{\nu}_s,\nu_s)|^2 \, \mathrm{d}s \right] +   \mathbb{E} \left[ \int_{0}^{t} |\tilde{\sigma}(Z^{\mu}_s,\mu_s) - \tilde{\sigma}(Z^{\nu}_s,\nu_s)|^2 \, \mathrm{d}s \right] \right) \\
& \leq C \Bigg( \mathbb{E} \left[ \int_{0}^{t} |\tilde{b}(Z^{\mu}_s,\mu_s) - \tilde{b}(Z^{\nu}_s,\mu_s)|^2 \, \mathrm{d}s \right] +  \mathbb{E} \left[ \int_{0}^{t} |\tilde{b}(Z^{\nu}_s,\mu_s) - \tilde{b}(Z^{\nu}_s,\nu_s)|^2 \, \mathrm{d}s \right] \\
& \quad +  \mathbb{E} \left[ \int_{0}^{t} |\tilde{\sigma}(Z^{\mu}_s,\mu_s) - \tilde{\sigma}(Z^{\nu}_s,\mu_s)|^2 \, \mathrm{d}s \right] +  \mathbb{E} \left[ \int_{0}^{t} |\tilde{\sigma}(Z^{\nu}_s,\mu_s) - \tilde{\sigma}(Z^{\nu}_s,\nu_s)|^2 \, \mathrm{d}s \right] \Bigg) \\
& \leq C \mathbb{E} \left[ \int_{0}^{t} \left( |Z^{\mu}_s-Z^{\nu}_s|^2 + \mathcal{W}_2^{2}(\mu_s,\nu_s) \right)  \mathrm{d}s \right] \leq C \int_{0}^{t}  \mathcal{W}_2^{2}(\mu_s,\nu_s)  \, \mathrm{d}s.
\end{align*}
For $k \geq 0$ and $t \in [0,T]$, we define the Picard iteration
\begin{equation}\label{eq:PIC}
\mu_t^{k+1} = \text{Law}\left(G^{-1}(Z_t^{\mu^{k}},\mu_t^{k}) \right),
\end{equation}
with $Z_t^{\mu^{0}} = G(\xi,\delta_{\xi})$ and $\mu_t^{0}=\delta_{\xi}$. Note that by It\^{o}'s formula (applied to $G^{-1}$), the process defined by $X_t^{k+1}:= G^{-1}(Z_t^{\mu^{k}},\mu_t^{k})$ is the solution to 
\begin{equation*}
\mathrm{d} X^{k+1}_t = b(X^{k+1}_t,\mu_t^{k}) \, \mathrm{d}t + \sigma(X^{k+1}_t) \, \mathrm{d}W_t, \quad  X^{k+1}_0= \xi.
\end{equation*}
Recall that $(X^{k+1}_t)_{t \in [0,T]}$ has uniformly bounded moments (uniformly in $k$), due to (H.\ref{Assum:A}(\ref{Assum:A1})) and (H.\ref{Assum:AA}(\ref{Assum:AA2})), i.e., we have 
\begin{equation}\label{eq:BMom}
\sup_{k \geq 1} \mathbb{E}\left[\sup_{0 \leq t \leq T} |X^{k}_t|^{p} \right] \leq C(1+\mathbb{E}[|\xi|^{p}]).
\end{equation}
 
The applicability of It\^{o}'s formula for $G^{-1}$ is a consequence of the fact that the inverse inherits the regularity of $G$, in particular the mapping $\mathcal{P}_2(\RR) \ni \mu \mapsto G^{-1}(y,\mu)$ is still an element of the class $\mathcal{C}^{(1,1)}_{b}$ (see, Proposition \ref{Prop:LD} in Appendix \ref{SEC:AA1}).
Then above estimate and Proposition \ref{Prop} yield
\begin{align}\label{eq:IT}
\sup_{t \in [0, T]} \mathcal{W}_2^{2}(\mu_t^{k+1},\mu_t^{k})  & \leq \sup_{t \in [0, T]}\mathbb{E}\left[|X^{k+1}_t-X^{k}_t|^2 \right] \notag \\
 & \leq 2 \sup_{t \in [0, T]} \mathbb{E}\left[ |G^{-1}(Z_t^{\mu^{k}},\mu_t^{k}) -  G^{-1}(Z_t^{\mu^{k-1}},\mu_t^{k})|^2 \right] \notag \\
& \quad + 2 \sup_{t \in [0, T]}\mathbb{E}\left[ |G^{-1}(Z_t^{\mu^{k-1}},\mu_t^{k}) -  G^{-1}(Z_t^{\mu^{k-1}},\mu_t^{k-1})|^2 \right] \notag \\
& \leq  C  \int_{0}^{T}  \mathcal{W}_2^{2}(\mu^{k}_s,\mu^{k-1}_s)   \mathrm{d}s + 2L(c) \sup_{t \in [0, T]} \mathcal{W}_2^{2}(\mu^{k}_{t},\mu^{k-1}_{t}) \notag \\
& \leq C\int_{0}^{T}  \mathcal{W}_2^{2}(\mu^{k}_s,\mu^{k-1}_s)   \mathrm{d}s +  2L(c) \sup_{t \in [0, T]} \mathbb{E}\left[|X^{k}_{t}-X^{k-1}_{t}|^2 \right],
\end{align}
where $L:= 2L(c) < 1$ due to the choice of $c$ and $C>0$ is a constant depending on the constants appearing in Proposition \ref{Prop} and Lemma \ref{lem:Lip}. 

Let now $0< T_0 <T$ such that $CT_0 + L < 1$. With this choice the uniqueness of (\ref{eq:Model2}) on $[0,T_0]$ follows from the estimate in (\ref{eq:IT}) by assuming there exist two solutions $(X,\mu)$ and $(Y,\nu)$ to (\ref{eq:Model2}), with $\mu_t$ and $\nu_t$ the marginal laws of $X_t$ and $Y_t$, respectively, for $t \in [0,T_0]$. In addition, we observe that the sequence of flows $(\mu^{k})_{k}$, for $\mu^{k}=(\mu^{k}_t)_{t \in [0,T_0]}$, is a Cauchy sequence in the complete metric space $\mathcal{C}([0,T],\mathcal{P}_2(\RR))$ equipped with the Wasserstein distance  $\sup_{t \in [0, T_0]} \mathcal{W}_2(\mu_t,\nu_t)$. Hence, \eqref{eq:PIC} has a fixed point, in particular we have 
$X_t = G^{-1}(Z_t^{\mu},\mu_t)$, where $\mu_t=\mathcal{L}_{X_t}$. It\^{o}'s formula applied to $G^{-1}$ yields the claim for the time interval $[0,T_0]$. Repeating the above procedure starting at $T_0$, we can extend the solution to the interval $[T_0,T_1]$, for some $T_0 < T_1 <T$. This is possible as the choice of $T_1$ depends on $X_{T_0}$ only through the second moment of $X_{T_0}$, for which we have a uniform bound for the entire interval $[0,T]$, see \eqref{eq:BMom}. Proceeding in such a manner, we can obtain well-posedness of (\ref{eq:Model2}) on $[0,T]$.
\end{proof}
\subsection{Interacting particle system with non-decomposable drift}\label{Sec:IAP}
In the following, we state the model assumptions which will specify the set-up for this subsection: 
\begin{Assumption}\label{Assum:AAA}
Assumptions (H.\ref{Assum:AA}(\ref{Assum:AA4})) and (H.\ref{Assum:AA}(\ref{Assum:AA5})) are satisfied and we require:
\begin{enumerate}[(1)]
    \item \label{Assum:AAb1} 
    Assumption (H.\ref{Assum:A}(\ref{Assum:A1})) holds and there exists a constant $L >0$ such that $|\sigma(x)| \leq L$ for all $x \in \mathbb{R}$.     
   \item \label{Assum:AAA1}
   There exists a constant $L_1>0$ such that  
     \begin{equation*}
      | b(x,\mu)  - b(x,\nu) | \leq L_1 \mathcal{W}_2(\mu,\nu) \quad \forall x \neq 0 \in \mathbb{R}, \ \forall \mu, \nu \in  \mathcal{P}_2(\mathbb{R}).
    \end{equation*}  
   Further, for any $\mu \in \mathcal{P}_2(\RR)$, $x \mapsto b(x,\mu)$ is piecewise Lipschitz continuous on the subintervals $(-\infty,0)$ and $(0,\infty)$, uniformly with respect to $\mu$.
     \item \label{Assum:AAA2} 
     $\alpha \in \mathcal{C}^{(1,2)}_{b}$ is a bounded function and the mappings
      \begin{align*}
        & \mathcal{P}_2(\RR) \times \RR \ni (\mu,y) \mapsto \partial_{\mu} \alpha(\mu)(y), \\
        & \mathcal{P}_2(\RR) \times \RR \ni (\mu,y) \mapsto \partial_y \partial_{\mu} \alpha(\mu)(y),\\
        & \mathcal{P}_2(\RR) \times \RR \times \RR \ni (\mu,y,y') \mapsto \partial^{2}_{\mu}\alpha(\mu)(y,y'),
      \end{align*}     
        are bounded and Lipschitz continuous. 
\end{enumerate} 
\end{Assumption}
\begin{remark}
Note that, compared to (H.\ref{Assum:AA}(\ref{Assum:AA2})), we do not require the drift to be uniformly bounded in the measure component.
\end{remark}
The interacting particles of the system $(X^{i,N}_t)_{t \in [0,T]}$, for $i \in \lbrace 1, \ldots, N \rbrace$ associated with (\ref{eq:Model2}) satisfy
\begin{equation}\label{eq:IPSystem2}
\mathrm{d}X_t^{i,N} = b(X_t^{i,N}, \mu_t^{\boldsymbol{X}^{N}}) \, \mathrm{d}t + \sigma(X_t^{i,N}) \, \mathrm{d}W_t^{i},
\end{equation} 
where $(\xi^{i},W^{i})$, for $i \in \lbrace 1, \ldots, N \rbrace$, are independent copies of $(\xi,W)$.

In contrast to the case of particle systems with decomposable drift, we set $\alpha: \mathcal{P}_2(\RR) \to\RR $,
\begin{equation*}
\alpha(\mu^{\boldsymbol{x}^{N}}) = \frac{b(0^{-},\mu^{\boldsymbol{x}^{N}})- b(0^{+},\mu^{\boldsymbol{x}^{N}})}{2\sigma^2(0)},
\end{equation*}
(which could also be interpreted as a mapping $\alpha_N: \RR^N \to \RR$) and apply to each particle the following transformation $G: \RR \times \mathcal{P}_2(\RR) \to\RR$,
\begin{equation}
\label{G_part}
G \left(x_i, \mu^{\boldsymbol{x}^{N}} \right)= x_i + \alpha\left( \mu^{\boldsymbol{x}^{N}}\right) x_i|x_i| \phi\big(x_i/c \big).
\end{equation}

We set $\RR^{N} \ni \boldsymbol{x}^{N} \mapsto G_i(\boldsymbol{x}^{N}) := G \left(x_i, \mu^{\boldsymbol{x}^{N}} \right)$, and use these mappings to define $\boldsymbol{G}_N: \RR^{N} \to \RR^{N}$ by 
\begin{equation*}
\boldsymbol{G}_N(\boldsymbol{x}^{N}) := \left(G_1(\boldsymbol{x}^{N}), \ldots, G_N(\boldsymbol{x}^{N}) \right)^{\top}.
\end{equation*}
To obtain the transformed process $(\boldsymbol{Z}^N_t)_{t \in [0,T]} = (Z^{1,N}_t,\ldots,Z^{N,N}_t)_{t \in [0,T]}^{\top} \in \RR^N$, we proceed as follows: For any $t \in [0,T]$ and $i \in \lbrace 1, \ldots, N \rbrace$, we have, using \cite[Proposition 5.35]{CD} (see also Section \ref{sec:Prelim})
\begin{align*} 
 \mathrm{d}G(X_t^{i,N},\mu_t^{\boldsymbol{X}^{N}}) &= \mathrm{d}G_i(X_t^{1,N},\ldots,X_t^{N,N}) \nonumber \\
&= \partial_{x_i} G(X_t^{i,N},\mu_t^{\boldsymbol{X}^{N}})\mathrm{d}X_t^{i,N}+\frac{1}{2} \partial^2_{x_i} G(X_t^{i,N},\mu_t^{\boldsymbol{X}^{N}})\mathrm{d}[X^{i,N}]_t \nonumber \\
& \quad + \frac{1}{N} \sum_{k =1}^{N} \partial_{\mu} G(X_t^{i,N},\mu_t^{\boldsymbol{X}^{N}})(X_t^{k,N}) \mathrm{d}X^{k,N}_t \nonumber \\
& \quad + \frac{1}{2N} \sum_{k =1}^{N} \partial_y \partial_{\mu}G(X_t^{i,N},\mu_t^{\boldsymbol{X}^{N}})(X_t^{k,N}) \mathrm{d}[X^{k,N}]_t \nonumber \\
& \quad + \frac{1}{2N^{2}} \sum_{k=1}^{N}\partial^2_{\mu}G(X_t^{i,N},\mu_t^{\boldsymbol{X}^{N}})(X_t^{k,N},X_t^{k,N}) \mathrm{d}[X^{k,N}]_t \nonumber \\
& \quad + \frac{1}{N} \partial_{x_i}  \partial_{\mu}G(X_t^{i,N},\mu_t^{\boldsymbol{X}^{N}})(X_t^{i,N}) \mathrm{d}[X^{i,N}]_t. \nonumber
\end{align*}
\noindent 
The applicability of It\^{o}'s formula for the function $G$ is guaranteed by \cite[Theorem 3.19]{sz2016b}. Note that Assumption 3.4 therein is imposed to guarantee Lipschitz continuity of second order derivatives of $G$ outside the set of discontinuities. Assumption (H.\ref{Assum:AAA}(\ref{Assum:AAA2})) and the definition of $G$ substitute this condition.   

Assuming for now the global invertibility of $\boldsymbol{G}_N$, we may introduce 
\begin{equation*}
\mathrm{d}\boldsymbol{Z}^{N}_t = \boldsymbol{B}_N\left(\boldsymbol{G}^{-1}_N(\boldsymbol{Z}^{N}_t) \right) \, \mathrm{d}t + \boldsymbol{\Sigma}_N\left(\boldsymbol{G}^{-1}_N(\boldsymbol{Z}^{N}_t) \right) \, \mathrm{d}\boldsymbol{W}^{N}_t, \quad \boldsymbol{Z}^{N}_0= \boldsymbol{G}_N((X_0^{1,N}, \ldots, X_0^{N,N})),
\end{equation*} 
where $\boldsymbol{W}^{N}_t = (W^{1}_t, \ldots, W^{N}_t)^{\top}$, $\boldsymbol{B}_N(\boldsymbol{x}^{N})=(B_1(\boldsymbol{x}^{N}), \ldots, B_N(\boldsymbol{x}^{N}))^{\top}$ is defined by
\begin{align}\label{Drift}
B_i(\boldsymbol{x}^{N})&:= \partial_{x_i} G(x_i,\mu^{\boldsymbol{x}^{N}}) b(x_i,\mu^{\boldsymbol{x}^{N}}) + \frac{1}{2} \sigma^{2}(x_i) \partial^2_{x_i} G(x_i,\mu^{\boldsymbol{x}^{N}}) \notag \\
& \quad +  \frac{1}{N} \sum_{k =1}^{N} \partial_{\mu} G(x_i,\mu^{\boldsymbol{x}^{N}})(x_k) b(x_k,\mu^{\boldsymbol{x}^{N}})  + \frac{1}{2N} \sum_{k =1}^{N} \partial_y \partial_{\mu}G(x_i,\mu^{\boldsymbol{x}^{N}})(x_k) \sigma^{2}(x_k) \nonumber \notag \\
& \quad + \frac{1}{2N^{2}} \sum_{k=1}^{N}\partial^2_{\mu}G(x_i,\mu^{\boldsymbol{x}^{N}})(x_k,x_k) \sigma^{2}(x_k) + \frac{1}{N} \partial_{x_i} \partial_{\mu}G(x_i,\mu^{\boldsymbol{x}^{N}})(x_i) \sigma^2(x_i), 
\end{align}
and $\boldsymbol{\Sigma}_N(\boldsymbol{x}^N) = \left(\Sigma^{i,j}(\boldsymbol{x}^N) \right)_{i,j \in \lbrace 1, \ldots, N \rbrace}$ by
\begin{align}\label{Diff}
\Sigma^{i,j}(\boldsymbol{x}^N) & = \partial_{x_i} G(x_i,\mu^{\boldsymbol{x}^{N}}) \sigma(x_i) \delta_{i,j} + \frac{1}{N}\partial_{\mu} G(x_i,\mu^{\boldsymbol{x}^{N}})(x_j) \sigma(x_j).
\end{align}
In the following lemma, we will prove the invertibility of $\boldsymbol{G}_N$. 
\begin{lemma}\label{lem:Inverse}
Let Assumption (H.\ref{Assum:AAA}(\ref{Assum:AAA2})) be satisfied and assume that the constant $c$ in \eqref{G_part} satisfies 
\begin{align}
\label{c_ineq}
c< \min \left(1, \left( \sup_{\boldsymbol{x}^N \in \RR^{N}} \left( |\alpha_N(\boldsymbol{x}^N)| + \max_{i \in \lbrace 1, \ldots, N \rbrace } |\partial_\mu \alpha(\mu^{\boldsymbol{x}^N})(x_i)| \right) \right)^{-1} \right).
\end{align}
Then, $\boldsymbol{G}_N$ has a global inverse.
\end{lemma}
\begin{proof}
We will employ Hadamard's global inverse function theorem (see, e.g., \cite[Theorem 2.2]{MRMS}) to prove that $\boldsymbol{G}_N$ has a global inverse. To do so, we need to verify the following properties of $\boldsymbol{G}_N$: $\boldsymbol{G}_N$ is in $\mathcal{C}^{1}(\RR^N,\RR^{N \times N})$, $\lim_{|\boldsymbol{x}^N| \to \infty} |\boldsymbol{G}_N(\boldsymbol{x}^N)| = \infty$, and $\boldsymbol{G}'_N(\boldsymbol{x}^N)$ is invertible for all $\boldsymbol{x}^N \in \RR^N$. The first two mentioned conditions are obvious, due to the definition of $\boldsymbol{G}_N$ and the uniform boundedness of $\alpha$ and $\bar{\phi}(x) :=  x|x| \phi\big(x/c \big)$. 

Hence, we need to prove that $\boldsymbol{G}'_N(\boldsymbol{x}^N)$ is invertible. First, note that 
\begin{equation*}
\boldsymbol{G}'_N(\boldsymbol{x}^N) = \mathrm{I}_{N \times N} + \text{diag}_{N \times N}(\bar{\phi}'(x_1)\alpha_N(\boldsymbol{x}^N), \ldots,\bar{\phi}'(x_N)\alpha_N(\boldsymbol{x}^N))+ \bar{\boldsymbol{\phi}}(\boldsymbol{x}^N)\boldsymbol{\alpha}'_N(\boldsymbol{x}^N),
\end{equation*} 
where $\mathrm{I}_{N \times N}$ is the $N \times N$ identity matrix, $\bar{\boldsymbol{\phi}}(\boldsymbol{x}^N) = (\bar{\phi}(x_1),\ldots,\bar{\phi}(x_N))^{\top}$, with $(\boldsymbol{\alpha}_N'(\boldsymbol{x}^N))_i= \frac{1}{N}\partial_\mu \alpha(\mu^{\boldsymbol{x}^N})(x_i)$ and $\boldsymbol{\alpha}'_N$ is a row vector.

Now, we define 
\begin{equation*}
\mathcal{A}(\boldsymbol{x}^N) :=  \text{diag}_{N \times N}(\bar{\phi}'(x_1)\alpha_N(\boldsymbol{x}^N), \ldots,\bar{\phi}'(x_N)\alpha_N(\boldsymbol{x}^N)) + \bar{\boldsymbol{\phi}}(\boldsymbol{x}^N)\boldsymbol{\alpha}'_N(\boldsymbol{x}^N),
\end{equation*} 
and remark that $\boldsymbol{G}'_N(\boldsymbol{x}^N)$ can be identified with the linear operator $\mathrm{I}_{N \times N} + \mathcal{A}(\boldsymbol{x}^{N}): \RR^N \to \RR^N$. Therefore, succeeding in showing that $c$ can be chosen (uniformly in $\boldsymbol{x}^{N}$) in a way such that the operator norm of $\mathcal{A}(\boldsymbol{x}^N)$ is smaller than one would yield the claim, as in this case $\boldsymbol{G}'_N(\boldsymbol{x}^N)$ is close to the identity (see, \cite[Lemma 3.17]{sz2016b}). We compute
\begin{align*}
\| \mathcal{A}(\boldsymbol{x}^N) \|  &  \leq   \max_{i \in \lbrace{1, \ldots, N \rbrace}} |\bar{\phi}'(x_i)| |\alpha_N(\boldsymbol{x}^N)|+  \max_{i \in \lbrace 1, \ldots, N \rbrace} |\bar{\phi}(x_i)|  |\partial_\mu \alpha(\mu^{\boldsymbol{x}^N})(x_i)| \\
& \leq  c  |\alpha_N(\boldsymbol{x}^N)| +  c^2 \max_{i \in \lbrace 1, \ldots, N \rbrace}|\partial_\mu \alpha(\mu^{\boldsymbol{x}^N})(x_i)|,
\end{align*}
which implies that for
\begin{align*}
c< \min \left(1, \left(|\alpha_N(\boldsymbol{x}^N)| + \max_{i \in \lbrace 1, \ldots, N \rbrace}|\partial_\mu \alpha(\mu^{\boldsymbol{x}^N})(x_i)| \right)^{-1} \right),
\end{align*}
$\| \mathcal{A}(\boldsymbol{x}^N) \| < 1$. Note that (H.\ref{Assum:AAA}(\ref{Assum:AAA2})) guarantees that $\alpha$ and its derivatives are uniformly bounded, i.e., $c$ can be chosen uniformly in $\boldsymbol{x}^N$.
Therefore, Hadamard's global inverse function theorem proves that, for each given $N  \geq 1$, $\boldsymbol{G}_N: \RR^N \to \RR^N$ is a diffeomorphism. 
\end{proof}

We proceed by showing that the transformed SDE has (locally) Lipschitz continuous coefficients. 
\begin{lemma}\label{lem:lip}
Let Assumption (H.\ref{Assum:AAA}) be satisfied. Then, the coefficients $\boldsymbol{B}_N$ and $\boldsymbol{\Sigma}_N$ introduced in (\ref{Drift}) and (\ref{Diff}), respectively, are locally Lipschitz continuous with linear growth.
\end{lemma}
\begin{proof}
For $\boldsymbol{x}^{N}, \boldsymbol{y}^{N} \in \RR^N$, we obtain using (\ref{Diff}) 
\begin{align*}
& \| \boldsymbol{\Sigma}_N(\boldsymbol{x}^{N}) - \boldsymbol{\Sigma}_N(\boldsymbol{y}^{N}) \|^2  \leq \sum_{i=1}^{N} |\Sigma^{i,i}(\boldsymbol{x}^{N}) - \Sigma^{i,i}(\boldsymbol{y}^{N})|^2 + \sum_{i \neq j} |\Sigma^{i,j}(\boldsymbol{x}^{N}) - \Sigma^{i,j}(\boldsymbol{y}^{N})|^2 \\
& \leq \sum_{i=1}^{N} |\partial_{x_i} G(x_i,\mu^{\boldsymbol{x}^{N}}) \sigma(x_i) - \partial_{y_i} G(y_i,\mu^{\boldsymbol{y}^{N}}) \sigma(y_i)|^2  \\
& \quad  + \frac{1}{N^2} \sum_{i \neq j} |\partial_{\mu} G(x_i,\mu^{\boldsymbol{x}^{N}})(x_j) \sigma(x_j) -\partial_{\mu} G(y_i,\mu^{\boldsymbol{x}^{N}})(y_j) \sigma(y_j)|^2 \\
& \leq C |\boldsymbol{x}^{N} - \boldsymbol{y}^{N}|^2,
\end{align*}
where we used (H.\ref{Assum:AAA}(\ref{Assum:AAb1})), and the Lipschitz continuity of the functions $x \mapsto \partial_{x} G(x,\mu)$ and $\mu \mapsto \partial_{x} G(x,\mu)$ to estimate the first term. Assumptions (H.\ref{Assum:AAA}(\ref{Assum:AAb1})) and (H.\ref{Assum:AAA}(\ref{Assum:AAA2})), in particular the Lipschitz continuity of $x \mapsto \partial_{\mu} G(x,\mu)(y)$ and $\mu \mapsto \partial_{\mu} G(x,\mu)(y)$, are employed to handle the second sum. Also note that all these mappings are bounded due to (H.\ref{Assum:AAA}(\ref{Assum:AAA2})) and the definition of the transformation. 

Noting that
\begin{align*}
& | \boldsymbol{B}_N(\boldsymbol{x}^{N}) - \boldsymbol{B}_N(\boldsymbol{y}^{N}) |^2 = \sum_{i=1}^{N} |B_i(\boldsymbol{x}^{N}) - B_i(\boldsymbol{y}^{N})|^2,
\end{align*}
we further obtain for the drift  
\begin{align*}
& |B_i(\boldsymbol{x}^{N}) - B_i(\boldsymbol{y}^{N})|^2 \\
& \leq C \Bigg( |\partial_{x_i} G(x_i,\mu^{\boldsymbol{x}^{N}}) b(x_i,\mu^{\boldsymbol{x}^{N}}) + \frac{1}{2} \sigma^{2}(x_i) \partial^2_{x_i} G(x_i,\mu^{\boldsymbol{x}^{N}}) \\
& \hspace{2cm}-\partial_{y_i} G(y_i,\mu^{\boldsymbol{y}^{N}}) b(y_i,\mu^{\boldsymbol{y}^{N}}) - \frac{1}{2} \sigma^{2}(y_i) \partial^2_{y_i} G(y_i,\mu^{\boldsymbol{y}^{N}})|^2  \\
& \quad +\frac{1}{N}\sum_{k =1}^{N}| \partial_{\mu} G(x_i,\mu^{\boldsymbol{x}^{N}})(x_k) b(x_k,\mu^{\boldsymbol{x}^{N}}) - \partial_{\mu} G(y_i,\mu^{\boldsymbol{y}^{N}})(y_k) b(y_k,\mu^{\boldsymbol{y}^{N}})|^2 \\
& \quad + \frac{1}{2N} \sum_{k =1}^{N} |\partial_y \partial_{\mu}G(x_i,\mu^{\boldsymbol{x}^{N}})(x_k) \sigma^{2}(x_k) - \partial_y \partial_{\mu}G(y_i,\mu^{\boldsymbol{x}^{N}})(y_k) \sigma^{2}(y_k)|^2  \\
& \quad + \frac{1}{2N^{2}} \sum_{k=1}^{N}|\partial^2_{\mu}G(x_i,\mu^{\boldsymbol{x}^{N}})(x_k,x_k) \sigma^{2}(x_k)-\partial^2_{\mu}G(y_i,\mu^{\boldsymbol{y}^{N}})(y_k,y_k) \sigma^{2}(y_k)|^2 \\
& \quad + \frac{1}{N} |\partial_{x_i} \partial_{\mu}G(x_i,\mu^{\boldsymbol{x}^{N}})(x_i) \sigma^2(x_i)-\partial_{y_i} \partial_{\mu}G(y_i,\mu^{\boldsymbol{y}^{N}})(y_i)) \sigma^2(y_i)|^2 \Bigg)=: \sum_{i=1}^{5} \Pi_i.
\end{align*}
That the terms $\Pi_3,\Pi_4$ and $\Pi_5$ allow a Lipschitz bound is a consequence of (H.\ref{Assum:AAA}(\ref{Assum:AAb1})) and (H.\ref{Assum:AAA}(\ref{Assum:AAA2})). 

For any $R>0$, in view of (H.\ref{Assum:AAA}(\ref{Assum:AAA1})) and (H.\ref{Assum:AAA}(\ref{Assum:AAA2})), we derive the following estimate for $\Pi_2$: 
\begin{align*}
\Pi_2 &\leq \frac{C}{N} \Bigg( \sum_{k = 1}^{N}  \left|  \partial_{\mu} \alpha(\mu^{\boldsymbol{x}^{N}})(x_k)\bar{\phi}(x_i) b(x_k,\mu^{\boldsymbol{x}^{N}}) -  \partial_{\mu} \alpha(\mu^{\boldsymbol{y}^{N}})(y_k)\bar{\phi}(x_i)b(y_k,\mu^{\boldsymbol{y}^{N}}) \right|^2  \\
& \quad +  \sum_{k = 1}^{N}  \left|  \partial_{\mu} \alpha(\mu^{\boldsymbol{y}^{N}})(y_k)\bar{\phi}(x_i)b(y_k,\mu^{\boldsymbol{y}^{N}}) -  \partial_{\mu} \alpha(\mu^{\boldsymbol{y}^{N}})(y_k)\bar{\phi}(y_i)b(y_k,\mu^{\boldsymbol{y}^{N}}) \right|^2 \Bigg) \\
& \leq \frac{C}{N} \sum_{k = 1}^{N}  \left|  \partial_{\mu} \alpha(\mu^{\boldsymbol{x}^{N}})(x_k) b(x_k,\mu^{\boldsymbol{x}^{N}}) -  \partial_{\mu} \alpha(\mu^{\boldsymbol{y}^{N}})(y_k)b(y_k,\mu^{\boldsymbol{y}^{N}}) \right|^2 + L_R|x_i-y_i|^2,
\end{align*}
for some constant $L_R>0$ and any $|\boldsymbol{x}^{N}|, |\boldsymbol{y}^{N}| \leq R$. We proceed with the estimate 
\begin{equation*}
\sum_{k = 1}^{N}  \left|  \partial_{\mu} \alpha(\mu^{\boldsymbol{x}^{N}})(x_k) b(x_k,\mu^{\boldsymbol{x}^{N}}) -  \partial_{\mu} \alpha(\mu^{\boldsymbol{y}^{N}})(y_k)b(y_k,\mu^{\boldsymbol{y}^{N}}) \right|^2 \leq C \sum_{k = 1}^{N} |x_k-y_k|^2,
\end{equation*}
which holds due to Assumptions (H.\ref{Assum:AA}(\ref{Assum:AA4})) and (H.\ref{Assum:AA}(\ref{Assum:AA5})).
Combining above estimates, we obtain
\begin{align*}
\Pi_2 \leq L_R \left( |x_i-y_i|^2 + \frac{1}{N}\sum_{k = 1}^{N} |x_k-y_k|^2 \right).
\end{align*}
Finally, we point out that
\begin{align*}
x \mapsto  b(x,\mu) + \frac{1}{2} \sigma^{2}(x) \partial^2_{x} G(x,\mu),
\end{align*}
is Lipschitz continuous due to the choice of $\alpha$. Employing this along with (H.\ref{Assum:AAA}(\ref{Assum:AAb1})), (H.\ref{Assum:AAA}(\ref{Assum:AAA1})) and (H.\ref{Assum:AAA}(\ref{Assum:AAA2})), we derive 
\begin{align*}
\Pi_1 \leq C \left( |x_i-y_i|^2 + \frac{1}{N}\sum_{k = 1}^{N} |x_k-y_k|^2 \right),
\end{align*}
for some constant $C >0$. Taking the estimates for $\Pi_1, \ldots, \Pi_5$ into account, yields the local Lipschitz continuity of $\boldsymbol{B}_N$. 

The linear growth of $\boldsymbol{B}_N$ and $\boldsymbol{\Sigma}_N$, i.e., that there exists a constant $C>0$ such that $|\boldsymbol{B}_N(\boldsymbol{x}^{N})| + \|\boldsymbol{\Sigma}_N(\boldsymbol{x}^{N})\| \leq C(1+|\boldsymbol{x}^{N}|)$ for all $\boldsymbol{x}^{N} \in \RR^{N}$, is a direct consequence of the growth conditions on $b$ and $\sigma$ along with the bounds for the derivatives of $G$ and $\alpha$.
\end{proof}

\begin{theorem}
Let Assumption (H.\ref{Assum:AAA}) be satisfied, let $\xi \in L_p^{0}(\mathbb{R})$ for a given $p \geq 2$ and assume that the constant $c$ in \eqref{G_part} satisfies \eqref{c_ineq}.
Then, the interacting particle system defined in (\ref{eq:IPSystem2}) has a unique strong solution in $\mathcal{S}^{p}([0,T])$.
\end{theorem}
\begin{proof}
From Lemma \ref{lem:lip} and the linear growth of $\boldsymbol{B}_N$ and $\boldsymbol{\Sigma}_N$, we can deduce that the SDE for $\boldsymbol{Z}^{N}$ has a unique strong solution (see, \cite[Chapter 5, Theorem 2.5]{XM}). Applying now It\^{o}'s formula to $\boldsymbol{G}_N^{-1}(\boldsymbol{Z}^{N}_t)$ proves the strong uniqueness of the particle system defined in \eqref{eq:IPSystem2}. Note that $\boldsymbol{G}_N^{-1}$ exists due to Lemma \ref{lem:Inverse} and It\^{o}'s formula is applicable for $\boldsymbol{G}_N^{-1}$ as it inherits the regularity of $\boldsymbol{G}_N$ (see, Appendix \ref{SEC:APP2} for details).
\end{proof}

\section{Euler--Maruyama scheme with and without transformation}\label{sec:isolatednumerics}

In this section, we restrict most of our discussion to the case of a McKean--Vlasov SDE with decomposable drift, due to the simpler structure of the underlying transformation and the particle systems. 

In the following subsections, we will present two Euler--Maruyama schemes to discretise the particle system defined in (\ref{eq:IPSystem}) in time. 

For the first scheme (Scheme 1), we will discretise the transformed (continuous) particle system in time and then exploit the global inverse $G^{-1}$ to obtain approximations of the original (discontinuous) particle system. A slight modification of this scheme will also be applied in the non-decomposable case. An approximation result with respect to the number of particles will also be presented.

The second scheme (Scheme 2) will be defined by directly discretising the discontinuous particle system, without making use of the transformation $G$.
We give strong convergence rates in terms of the number of time-steps and pathwise strong propagation of chaos results in order to obtain quantitative $L_2$-approximations for the underlying McKean--Vlasov SDE. 

\subsection{Scheme 1: Euler--Maruyama after transformation (decomposable case)}

We define the following explicit Euler--Maruyama scheme to discretise the particle system (\ref{eq:IPSystem}) in time. In a first step, we partition a given time interval $[0,T]$ into subintervals of equal length $h=T/M$, for some integer $M>0$, and define $t_n:= nh$. Then, we simulate the transformed particle system by
\begin{equation}\label{eq:Euler}
Z_{t_{n+1}}^{i,N,M} = Z_{t_{n}}^{i,N,M} + \tilde{b}(G^{-1}(Z_{t_n}^{i,N,M}), \mu_{t_n}^{\boldsymbol{Z}^{N,M}}) h + \tilde{\sigma}(G^{-1}(Z_{t_n}^{i,N,M})) \Delta W_{n}^{i}, 
\end{equation}
for $n \in \lbrace 0, \ldots, M-1 \rbrace$, where $Z_0^{i,N,M} = G(X_0^{i,N})$, $\Delta W_{n}^{i} = W_{t_{n+1}}^{i} - W_{t_{n}}^{i}$, for $i \in \lbrace 1, \ldots, N \rbrace$, and
\begin{equation*}
\mu_{t_n}^{\boldsymbol{Z}^{N,M}}(\mathrm{d}x) := \frac{1}{N} \sum_{j=1}^{N} \delta_{G^{-1}(Z_{t_{n}}^{j,N})}(\mathrm{d}x).
\end{equation*} 
We introduce the notation $\eta(t):= \sup \lbrace s \in \lbrace 0, h, \ldots, Mh \rbrace : s \leq t \rbrace$, for $t \in [0,T]$, which allows us to define the continuous time version of (\ref{eq:Euler})
\begin{equation}\label{eq:Euler2}
Z_{t}^{i,N,M} = Z_0^{i,N,M} + \int_{0}^{t} \tilde{b}(G^{-1}(Z_{\eta(s)}^{i,N,M}), \mu_{\eta(s)}^{\boldsymbol{Z}^{N,M}}) \, \mathrm{d}s + \int_{0}^{t} \tilde{\sigma}(G^{-1}(Z_{\eta(s)}^{i,N,M})) \, \mathrm{d}W_s^{i}.
\end{equation}
Then, we propose an Euler--Maruyama approximation to $X_t^{i,N}$, for $i \in \lbrace 1, \ldots, N \rbrace$ and $t \in [0,T]$, by
\begin{eqnarray}
\label{trans_EM}
X_t^{i,N,M}=G^{-1}(Z_{t}^{i,N,M}).
\end{eqnarray}
\noindent
The convergence of this algorithm is proven in the following theorem:
\begin{theorem}
Let Assumption (H.\ref{Assum:A}) be satisfied, let $\xi \in L_p^{0}(\mathbb{R})$ for some $p > 4$ and assume $c < 1/|\alpha|$. For $i \in \lbrace 1, \ldots, N \rbrace$, let $(X_t^{i})_{t \in [0,T]}$ be the unique strong solution of (\ref{eq:Model1}) driven by the Brownian motion $(W_t^{i})_{t \in [0,T]}$ with initial data $\xi^{i}$, and $(X_t^{i,N,M})_{t \in [0,T]}$ be given by \eqref{eq:Euler2} and \eqref{trans_EM}. 
Then, there exists a constant $C>0$ (independent of $N$ and $M$) such that 
\begin{equation*}
\max_{i \in \lbrace 1, \ldots, N \rbrace}\mathbb{E}\left[\sup_{t \in [0,T]} |X_t^{i} - X_t^{i,N,M}|^2 \right] \leq C \left(h + N^{-1/2} \right).
\end{equation*}
\end{theorem}
\begin{proof}
Note that
\begin{equation*}
|X_t^{i} - X_t^{i,N,M}|^2 \leq 2 |X_t^{i} - X_t^{i,N}|^2 + 2|X_t^{i,N} - X_t^{i,N,M}|^2,
\end{equation*}
and recall that the dynamics of $Z_t^{i} = G(X_t^{i})$ satisfies
\begin{equation*}
\mathrm{d}Z_t^{i} = \tilde{b}(Z_t^{i}, \tilde{\mu}_t^{Z}) \, \mathrm{d}t + \tilde{\sigma}(Z_t^{i}) \, \mathrm{d}W_t^{i},
\end{equation*} 
where $\tilde{\mu}_t^{Z} := \mathcal{L}_{G^{-1}(Z^{i}_t)}$. Therefore, we obtain for some constant $C>0$
\begin{align*}
\mathbb{E} \left[\sup_{t \in [0,T]} |X_t^{i} - X_t^{i,N}|^2 \right] &= \mathbb{E} \left[\sup_{t \in [0,T]} |G^{-1}(Z_t^{i}) - G^{-1}(Z_t^{i,N})|^2 \right] \\
& \leq L_{G^{-1}}^2 \mathbb{E} \left[\sup_{t \in [0,T]} |Z_t^{i} - Z_t^{i,N}|^2 \right] \leq C N^{-1/2},
\end{align*}
where the last inequality can be derived similarly to the propagation of chaos results for equations with Lipschitz continuous coefficients as in, e.g., \cite{RC} for $d=1$ (note that the rate $N^{-1/2}$ can be improved into $N^{-1}$ in case where $b_2(x,\mu) = \int_{\mathbb{R}} \beta(x,y) \, \mu(\mathrm{d}y)$, with $\beta$ Lipschitz continuous, see \cite{MEL}). To apply the aforementioned propagation of chaos result, we need the requirement that $\xi \in L_p^{0}(\RR)$ for $p > 4$ (see, also \cite[Theorem 5.8]{CD}). 
Furthermore, we have
\begin{align*}
\mathbb{E} \left[\sup_{t \in [0,T]} |X_t^{i,N} - X_t^{i,N,M}|^2 \right] &= \mathbb{E} \left[\sup_{t \in [0,T]} |G^{-1}(Z_t^{i,N}) - G^{-1}(Z_t^{i,N,M})|^2 \right] \\
& \leq L_{G^{-1}}^2 \mathbb{E} \left[\sup_{t \in [0,T]} |Z_t^{i,N} - Z_t^{i,N,M}|^2 \right] \leq  C h,
\end{align*}
since the SDEs for $(Z_t^{i,N})_{t \in [0,T]}$ and $(Z_t^{i,N,M})_{t \in [0,T]}$ have globally Lipschitz continuous coefficients. From these two estimates the claim follows.
\end{proof}

\subsection{Scheme 2: Euler--Maruyama without transformation (decomposable case)}
As $G$ and $G^{-1}$ may be difficult to construct in multi-dimensional settings, and since the evaluation for the inverse at each time point can be computationally expensive, it would be preferable to discretise the particle system $(X^{i,N})_{i \in \lbrace 1, \ldots, N \rbrace}$ in time directly, without the use of the transformation $G$. In addition, a drawback of Scheme 1 is that an SDE with additive diffusion term will be transformed into one with multiplicative noise and therefore the Euler--Maruyama scheme no longer coincides with the Milstein scheme. 
We employ an Euler--Maruyama scheme to discretise the particle system (\ref{eq:IPSystem}) and compute an approximate solution by
\begin{equation}\label{eq:EulerX}
X_{t}^{i,N,M} = X^{i,N}_0 + \int_{0}^{t} \left( b_1(X_{\eta(s)}^{i,N,M}) + b_2(X_{\eta(s)}^{i,N,M}, \mu_{\eta(s)}^{\boldsymbol{X}^{N,M}}) \right) \, \mathrm{d}s + \int_{0}^{t} \sigma(X_{\eta(s)}^{i,N,M}) \, \mathrm{d}W_s^{i}.
\end{equation}
The following results are concerned with moment stability of the discretised particle system and estimates for the occupation time of the particle system in the neighbourhood of the set of discontinuities. \\    
\noindent
\textbf{Moment stability:} \\ 
\noindent
We first remark that due to the linear growth of the coefficients in the state component, the Lipschitz continuity in the measure variable and the fact that all particles are identically distributed, we have the following result (see, e.g., \cite{XM} for details): 
\begin{proposition}\label{Prop:Moment}
Let Assumption (H.\ref{Assum:A}) be satisfied, and let $\xi \in L_p^{0}(\mathbb{R})$ for $p \geq 2$. Then, there exist constants $C_1, C_2>0$, such that 
\begin{equation*}
\max_{i \in \lbrace 1, \ldots, N \rbrace} \max_{n \in \lbrace 0, \ldots, M \rbrace} \mathbb{E}\left[ |X_{t_n}^{i,N,M} |^p \right] \leq C_1,
\end{equation*}
and for all $i \in \lbrace 1, \ldots, N \rbrace$ and for all $t \in [0,T]$,
\begin{equation*}
\mathbb{E} \left[ |X_{t}^{i,N,M} - X_{\eta(t)}^{i,N,M}|^p\right] \leq C_2h^{p/2}.
\end{equation*}
\end{proposition}  
\noindent
\textbf{Occupation time formula for (\ref{eq:EulerX}):} \\ 
\noindent
Below, we will show an estimate of the expected occupation time of a fixed particle of the system defined by (\ref{eq:EulerX}) in a neighbourhood of zero. 
\begin{proposition}\label{Prop:OTParticle2}
Let Assumption (H.\ref{Assum:A}) be satisfied and let $i$ be an arbitrary but fixed particle index. Further, let $\xi \in L_p^{0}(\RR)$ for $p \geq 2$ and let $(X_t^{i,N,M})_{t \in [0,T]}$ be given by (\ref{eq:EulerX}). Then, there exists a constant $C>0$ such that for all $N, M \in \mathbb{N}$ and all sufficiently small $\varepsilon>0$ 
\begin{equation*}
\int_{0}^{T} \mathbb{P} \left( \lbrace \boldsymbol{X}_{t}^{N,M} \in \Theta^{i,\varepsilon} \rbrace \right) \, \mathrm{d}t \leq C \varepsilon,
\end{equation*}
where $ \boldsymbol{X}_{t}^{N,M} = (X_{t}^{1,N,M},\ldots,X_{t}^{N,N,M})^{\top} \in \RR^{N}$ and $\Theta^{i,\varepsilon}$ is given by
\begin{equation*}
\Theta^{i,\varepsilon}:= \lbrace \boldsymbol{x}^{N} =(x_1, \ldots, x_N)^{\top} \in \RR^N: \ \exists \boldsymbol{y}^N \in \Theta^i \text{ with }  |\boldsymbol{x}^N-\boldsymbol{y}^N| < \varepsilon \rbrace,
\end{equation*}
with $\Theta^i:=\{\boldsymbol{x}^{N}=(x_1,\ldots,x_N)^{\top} \in \RR^{N} \ \colon \ x_i=0\}$.
\end{proposition} 
\begin{proof}
We aim to apply \cite[Theorem 2.7]{sz2016c}, which states the following: Let $(X_t)_{t \in [0,T]}$ be an $\RR^d$-valued It\^{o} process 
\begin{equation*}
X_T = X_0 + \int_{0}^{T} A_t \, \mathrm{d}t + \int_{0}^{T} B_t \, \mathrm{d}W_t,
\end{equation*}
with progressively measurable processes $A=(A_t)_{t \in [0,T]}$ and $B=(B_t)_{t \in [0,T]}$, where $A$ is $\RR^d$-valued and $B$ is $\RR^{d \times d}$-valued. The set of discontinuities, $\Theta$, is assumed to be a $\mathcal{C}^3$ hypersurface of positive reach. Namely, there exists $ \varepsilon > 0$ such that $p(x)= \argmin_{y \in \Theta} |x-y|$ is a single valued function of class $\mathcal{C}^3$ on the tubular neighbourhood $\Theta^{\varepsilon}:= \lbrace x \in \RR^{d}: \ \inf_{y \in \Theta} |x-y| < \varepsilon \rbrace$ (see Definition 2.4 in \cite{sz2016c} for details).
Then, there exists a constant $C>0$, such that for all sufficiently small $\varepsilon>0$ 
\begin{equation*}
\int_{0}^{T} \mathbb{P} \left( \lbrace X_t \in \Theta^{\varepsilon} \rbrace  \right) \, \mathrm{d}t \leq C \varepsilon,
\end{equation*}
assuming, additionally, that 
\begin{enumerate}
\item the processes $A$ and $B$ are almost surely bounded by a constant $C$ if $X_t(\omega)$ is in a small neighbourhood of $\Theta$, and
\item there exists a constant $C>0$ such that for almost all $\omega \in \Omega$, we have: If, for any $t \in [0,T]$, $X_t(\omega)$ is in a small neighbourhood of $\Theta$ then   
\begin{equation*}
n^{\top}\left(p\left( X_t(\omega) \right) \right) B_t^{\top}(\omega)B_t(\omega) n\left(p\left( X_t(\omega) \right) \right) \geq C,
\end{equation*}
where $n(x)$ has length one and is orthogonal to the tangent space of $\Theta$ in $x$.
\end{enumerate}
\noindent
We now return to our particular model problem. First, we remark that $\Theta^{i}$ satisfies all regularity conditions of \cite[Theorem 2.7]{sz2016c}, i.e., it is a $\mathcal{C}^{3}$ hypersurface of positive reach. We then observe that the $N$-dimensional particle system can be rewritten as
\begin{equation*}
\mathrm{d}\boldsymbol{X}_{t}^{N,M} = \boldsymbol{B}_N(\boldsymbol{X}_{\eta(t)}^{N,M}) \, \mathrm{d}t + \boldsymbol{\Sigma}_N(\boldsymbol{X}_{\eta(t)}^{N,M}) \, \mathrm{d}\boldsymbol{W}^{N}_t,
\end{equation*}
where $\boldsymbol{W}^{N}_t= (W^{1}_t, \ldots, W^{N}_t)^{\top}$ and $\boldsymbol{B}_N: \RR^N \to \RR^N$ and $\boldsymbol{\Sigma}_N: \RR^N \to \RR^{N \times N}$ are defined by
\begin{align*}
\boldsymbol{B}_N(\boldsymbol{x}^{N}) & = (b(x_1,\mu^{\boldsymbol{x}^{N}}),\ldots,b(x_N,\mu^{\boldsymbol{x}^{N}}))^{\top}, \\
\boldsymbol{\Sigma}_N(\boldsymbol{x}^N) & = \mbox{diag}_{N \times N}(\sigma(x_1),\ldots,\sigma(x_N)).
\end{align*}
Further, we observe that there is a constant $C>0$ such that: If, for any $t \in [0,T]$ and $\omega \in \Omega$, $\boldsymbol{X}_{t}^{N,M}(\omega)$ is in a small neighbourhood of $\Theta^{i}$ then  
\begin{equation*}
n^{\top}\left(p\left( \boldsymbol{X}_{t}^{N,M}(\omega) \right) \right) \boldsymbol{\Sigma}_N^{\top}(\boldsymbol{X}_{t}^{N,M}(\omega))\boldsymbol{\Sigma}_N(\boldsymbol{X}_{t}^{N,M}(\omega)) n\left(p\left( \boldsymbol{X}_{t}^{N,M}(\omega) \right) \right) \geq C,
\end{equation*}
as $\sigma$ is continuous and $\sigma(0) \neq 0$. Also, note that a normal vector of the tangent space of $\Theta^{i}$ is $e_i$, i.e., the $i$-th unit vector. Further, a close inspection of the proof of \cite[Theorem 2.7]{sz2016c}, shows that the boundedness assumption on the coefficients in a neighbourhood of $\Theta^{i}$ is not needed in our case, due to the moment bound of an individual particle established in Proposition \ref{Prop:Moment}. Hence,
\begin{equation*}
\int_{0}^{T} \mathbb{P} \left( \lbrace \boldsymbol{X}_{t}^{N,M} \in \Theta^{i,\varepsilon} \rbrace \right) \, \mathrm{d}t \leq C \varepsilon,
\end{equation*}
where the constant $C>0$ is independent of $N$, due to the fact that the normal vector is the $i$-th unit vector. 
\end{proof}
\noindent
\textbf{Auxiliary proposition:} \\ 
\noindent
Based on the occupation time estimate from Proposition \ref{Prop:OTParticle2}, we will prove the following result, which is needed in the proof of Theorem \ref{Prop:Conver} given below. 
\begin{proposition}\label{Prop:aux}
Let Assumption (H.\ref{Assum:A}) be satisfied, and let $\xi \in L_p^{0}(\RR)$ for $p \geq 8$. Furthermore, let $(X_t^{i,N,M})_{t \in [0,T]}$ be given by (\ref{eq:EulerX}).
Then, there exists a constant $C>0$ (independent of $N$ and $M$) such that for any $t \in [0,T]$, we have
\begin{align*}
\max_{i \in \lbrace 1, \ldots, N \rbrace} \mathbb{E}\left[ \left| \int_{0}^{t} \left( G''(X_s^{i,N,M}) - G''(X_{\eta(s)}^{i,N,M}) \right)\sigma^2(X_{\eta(s)}^{i,N,M}) \, \mathrm{d}s \right|^2 \right] \leq C h^{2/9}.
\end{align*}
\end{proposition}
\begin{proof}
First, observe that the linear growth of $\sigma$ and the piecewise Lipschitz continuity of $G''$ imply that there exists a constant $C>0$ such that
\begin{align*}
& \left| \left( G''(X_s^{i,N,M}) - G''(X_{\eta(s)}^{i,N,M}) \right)\sigma^2(X_{\eta(s)}^{i,N,M}) \right| \\
& \leq 
\begin{cases}
C \left(1+ (X_{\eta(s)}^{i,N,M})^2 \right) |X_s^{i,N,M} - X_{\eta(s)}^{i,N,M}|, \ X_s^{i,N,M} \notin (-\varepsilon,\varepsilon), \ |X_s^{i,N,M} - X_{\eta(s)}^{i,N,M}| < \varepsilon,  \nonumber \\
C \left(1+ (X_{\eta(s)}^{i,N,M})^2 \right), \text{ otherwise},
\end{cases}
\end{align*} 
where $\varepsilon>0$ will be specified later. With this at hand, we derive
\begin{align*}
& \mathbb{E}\left[ \left| \int_{0}^{t} \left( G''(X_s^{i,N,M}) - G''(X_{\eta(s)}^{i,N,M}) \right)\sigma^2(X_{\eta(s)}^{i,N,M}) \, \mathrm{d}s \right|^2 \right]  \\
& \leq  C \int_{0}^{t} \mathbb{E} \left[ \left| \left( G''(X_s^{i,N,M}) - G''(X_{\eta(s)}^{i,N,M}) \right)\sigma^2(X_{\eta(s)}^{i,N,M}) \right|^2 \right] \, \mathrm{d}s \\
& \leq  C \Bigg( \int_{0}^{t} \mathbb{E} \left[ \left| \left( G''(X_s^{i,N,M}) - G''(X_{\eta(s)}^{i,N,M}) \right)\sigma^2(X_{\eta(s)}^{i,N,M}) \right|^2 \left(\mathrm{I}_{\lbrace X_s^{i,N,M} \notin (-\varepsilon,\varepsilon) \rbrace } \mathrm{I}_{\lbrace |X_s^{i,N,M} - X_{\eta(s)}^{i,N,M}| < \varepsilon \rbrace } \right) \right] \, \mathrm{d}s  \\
& \quad +  \int_{0}^{t} \mathbb{E} \Big[ \left| \left( G''(X_s^{i,N,M}) - G''(X_{\eta(s)}^{i,N,M}) \right)\sigma^2(X_{\eta(s)}^{i,N,M}) \right|^2  \\
& \hspace{1.5cm}  \times \left(\mathrm{I}_{\lbrace X_s^{i,N,M} \in (-\varepsilon,\varepsilon) \rbrace } + \mathrm{I}_{\lbrace X_s^{i,N,M} \notin (-\varepsilon,\varepsilon) \rbrace } \mathrm{I}_{\lbrace |X_s^{i,N,M} - X_{\eta(s)}^{i,N,M}| \geq \varepsilon \rbrace }   \right) \Big] \, \mathrm{d}s \Bigg) \\ 
& \leq  C \Bigg( \int_{0}^{t} \mathbb{E} \left[ \left(1+ (X_{\eta(s)}^{i,N,M})^4 \right) |X_s^{i,N,M} - X_{\eta(s)}^{i,N,M}|^2 \left(\mathrm{I}_{\lbrace X_s^{i,N,M} \notin (-\varepsilon,\varepsilon) \rbrace } \mathrm{I}_{\lbrace |X_s^{i,N,M} - X_{\eta(s)}^{i,N,M}| < \varepsilon \rbrace } \right) \right] \, \mathrm{d}s  \\
& \quad + \int_{0}^{t} \mathbb{E} \left[ \left(1+ (X_{\eta(s)}^{i,N,M})^4 \right) \left(\mathrm{I}_{\lbrace X_s^{i,N,M} \in (-\varepsilon,\varepsilon) \rbrace } + \mathrm{I}_{\lbrace X_s^{i,N,M} \notin (-\varepsilon,\varepsilon) \rbrace } \mathrm{I}_{\lbrace |X_s^{i,N,M} - X_{\eta(s)}^{i,N,M}| \geq \varepsilon \rbrace }   \right) \right] \, \mathrm{d}s \Bigg) \\ 
& \leq C \left(\varepsilon^2 + \varepsilon^{1/2} + \int_{0}^{t} \left(\mathbb{P}(|X_s^{i,N,M} - X_{\eta(s)}^{i,N,M}| \geq \varepsilon) \right)^{1/2} \, \mathrm{d}s \right),
\end{align*}
where we used H\"{o}lder's inequality, Proposition \ref{Prop:Moment} and Proposition \ref{Prop:OTParticle2} in the last display. Markov's inequality along with Proposition \ref{Prop:Moment} imply that there exists a constant $C>0$ such that
\begin{equation*}
\left(\mathbb{P}(|X_s^{i,N,M} - X_{\eta(s)}^{i,N,M}| \geq \varepsilon) \right)^{1/2} \leq \frac{\left(\mathbb{E} \left[ \left|X_s^{i,N,M} - X_{\eta(s)}^{i,N,M} \right|^8 \right] \right)^{1/2}}{\varepsilon^4} \leq \frac{C h^2}{\varepsilon^4}.
\end{equation*}
Choosing $\varepsilon = h^{4/9}$ gives the result.
\end{proof}
\noindent
We are now ready to present our main convergence result. In this case, we only obtain the strong convergence rate of order $1/9$:

\begin{theorem}\label{Prop:Conver}
Let Assumption (H.\ref{Assum:A}) be satisfied, let $\xi \in L_p^{0}(\mathbb{R})$ for some $p \geq 8$ and assume $c < 1/|\alpha|$. Furthermore, let $(X_t^{i})_{t \in [0,T]}$ be the unique strong solution of (\ref{eq:Model1}) driven by the Brownian motion $(W_t^{i})_{t \in [0,T]}$with initial data $\xi^{i}$ and $(X_t^{i,N,M})_{t \in [0,T]}$ given by (\ref{eq:EulerX}). Then, there exists a constant $C>0$ (independent of $N$ and $M$) such that 
\begin{equation*}
\max_{i \in \lbrace 1, \ldots, N \rbrace} \mathbb{E}\left[\sup_{t \in [0,T]} |X_t^{i} - X_t^{i,N,M}|^2 \right] \leq C \left(h^{2/9}
+ N^{-1/2} \right).
\end{equation*}
\end{theorem}
\begin{proof}
Note that 
\begin{align}\label{eq:est2}
\mathbb{E}\left[\sup_{t \in [0,T]} |X_t^{i} - X_t^{i,N} |^2 \right] & \leq L_{G^{-1}} \mathbb{E} \left[  \sup_{t \in [0,T]} |G( X_t^{i}) - G(X_t^{i,N})|^2 \right] \nonumber \\ 
& \leq C N^{-1/2},
\end{align}
where in the last display, we used the pathwise propagation of chaos result as in the previous subsection.
Further, we have
\begin{align*}
\mathbb{E} \left[ \sup_{t \in [0,T]} |X_t^{i,N}-X_t^{i,N,M}|^2 \right] &\leq L_{G^{-1}} \mathbb{E} \left[  \sup_{t \in [0,T]} |Z_t^{i,N} - G(X_t^{i,N,M})|^2 \right] \\
& \leq C \left(  \mathbb{E} \left[  \sup_{t \in [0,T]} |Z_t^{i,N}-Z_t^{i,N,M}|^2 \right] +  \mathbb{E} \left[  \sup_{t \in [0,T]} |Z_t^{i,N,M} - G(X_t^{i,N,M})|^2 \right]  \right) \\
& \leq C \left( h + \mathbb{E} \left[  \sup_{t \in [0,T]} |Z_t^{i,N,M} - G(X_t^{i,N,M})|^2 \right]  \right),
\end{align*}
where in the last estimate, we employed standard strong convergence results for the Euler--Maruyama scheme applied to SDEs with globally Lipschitz continuous coefficients. 
Following similar arguments to \cite{sz2016c} or \cite{MGLY19}, one further obtains, by applying It\^{o}'s formula to $G(X_t^{i,N,M})$,
\begin{align*}
\mathbb{E} \left[  \sup_{t \in [0,T]} |Z_t^{i,N,M} - G(X_t^{i,N,M})|^2 \right] &\leq C \Bigg(h +  \int_{0}^{T} \mathbb{E}\left[ \left| \left( G''(X_s^{i,N,M}) - G''(X_{\eta(s)}^{i,N,M}) \right)\sigma^2(X_{\eta(s)}^{i,N,M}) \right|^2 \right]  \, \mathrm{d}s \\
& \qquad + \int_{0}^{T} \mathbb{E} \left[  \sup_{s \in [0,t]} |Z_s^{i,N,M} - G(X_s^{i,N,M})|^2 \right] \, \mathrm{d}t \Bigg),
\end{align*}
where the second summand on the right side is of order $h^{2/9}$ due to Proposition \ref{Prop:aux}. Hence, Gronwall's inequality yields 
\begin{equation}\label{eq:est1}
\mathbb{E} \left[ \sup_{t \in [0,T]} |X_t^{i,N}-X_t^{i,N,M}|^2 \right] \leq Ch^{2/9},
\end{equation}
and the claim follows combining (\ref{eq:est2}) and (\ref{eq:est1}).
\end{proof}
\noindent 
\begin{remark}
The convergence rate in terms of the number of particles in the above theorem can again be improved to $1/2$ if the drift has the form \eqref{eq:LIN}, see \cite{MEL}. 

The convergence rate in terms of number of time-steps established in Theorem \ref{Prop:Conver} could be improved by employing exponential tail estimate techniques, as in \cite{sz2016c}. The resulting strong convergence rate would be $1/4-\varepsilon$, for an arbitrarily small $\varepsilon>0$. However, to achieve this, one would need to assume boundedness of the coefficients in equation (\ref{eq:Model1}). Another possibility to recover a better convergence rate in our setting would be to require that the initial data $\xi \in L_p^{0}(\RR)$ for any $p \geq 2$ and that $\sigma$ is uniformly bounded. This would enable us to obtain sharper estimates, when employing Markov's inequality in the proof of Proposition \ref{Prop:aux}. If we assume moment boundedness of the initial data of any order, but allow $\sigma$ to grow-linearly, we would obtain a rate of order $1/8-\varepsilon$.

Moreover, although we expect
that the optimal convergence rate of the Euler--Maruyama scheme applied to the interacting particle system is $1/2$ (as for equations with Lipschitz coefficients), we only achieve the order $1/9$ (or $1/4-\varepsilon$ under stronger assumptions on the initial data or the coefficients of the underlying McKean--Vlasov SDE), due to the estimate of the probability that $X_{\eta(s)}^{i,N,M}$ and $X_s^{i,N,M}$ have a different sign, i.e., that the term $|G''(X_s^{i,N,M}) - G''(X_{\eta(s)}^{i,N,M})|$ in the proof of Proposition \ref{Prop:aux} does not allow a Lipschitz type estimate. Refined estimates of the aforementioned expected sign change, as derived in \cite{MGLY19} for one-dimensional SDEs, are not easy to prove for an interacting particle system. The proof of the main result in \cite{MGLY19} is not applicable to our setting as an individual particle (seen as a one-dimensional equation) does not satisfy Markov properties (due to the dependence of interaction terms), which are key in \cite{MGLY19}.    
\end{remark}

\subsection{Scheme 1 for the non-decomposable case}

Here, we first prove a propagation of chaos result in the case of non-decomposable drifts in Lemma \ref{PCLem}. The time-discretisation error is then given in Theorem \ref{TH:E3}. 
\begin{lemma}\label{PCLem}
Let Assumption (H.\ref{Assum:AAA}) hold, let $\xi \in L_p^{0}(\RR)$ for $p > 4$ and assume the mapping $\mathbb{R} \setminus \lbrace{ 0 \rbrace} \times \mathcal{P}_2(\RR) \ni (x,\mu) \mapsto b(x,\mu)$ to be uniformly bounded. Let $(X_t^{i})_{t \in [0,T]}$ be the unique strong solution of (\ref{eq:Model2}) driven by the Brownian motion $(W_t^{i})_{t \in [0,T]}$ with initial data $\xi^{i}$, and $(X_t^{i,N})_{t \in [0,T]}$ is the solution to the associated particle system. Then, there exists a constant $C>0$ (independent of $N$) such that 
\begin{equation*}
\max_{i \in \lbrace 1, \ldots, N \rbrace}\mathbb{E}\left[\sup_{t \in [0,T]} |X_t^{i} - X_t^{i,N}|^2 \right] \leq C N^{-1/2}.
\end{equation*}
\end{lemma}
\begin{proof}
First, we observe, using the definitions $\mu_t = \mathcal{L}_{X_t}$, $\mu^{N}_t(\mathrm{d}x) = \frac{1}{N} \sum_{j=1}^{N} \delta_{X_t^{j}}(\mathrm{d}x)$, and $\mu_t^{\boldsymbol{X}^{N},N-1}(\mathrm{d}x) = \frac{1}{N-1} \sum_{j \neq i} \delta_{X_t^{j,N}}(\mathrm{d}x)$, that 
\begin{align}\label{a1}
\mathbb{E} \left[\sup_{t \in [0,T]} |X_t^{i} - X_t^{i,N}|^2 \right] &= \mathbb{E} \left[\sup_{t \in [0,T]} |G^{-1}(G(X^{i}_t,\mu_t),\mu_t) - G^{-1}(G(X_t^{i,N},\mu_t^{\boldsymbol{X}^{N},N-1}),\mu_t^{\boldsymbol{X}^{N},N-1})|^2 \right] \notag \\
& \leq  C \Bigg( \mathbb{E} \left[\sup_{t \in [0,T]} |G(X^{i}_t,\mu_t)- G(X_t^{i,N},\mu_t^{\boldsymbol{X}^{N},N-1})|^2 \right] \notag \\
& \quad + \sup_{0 \leq t \leq T} L(c)\left( \mathcal{W}^{2}_2(\mu_t,\mu^{N}_t) +   \mathcal{W}^{2}_2(\mu^{N}_t,\mu_t^{\boldsymbol{X}^{N},N-1})  \right) \Bigg),
\end{align} 
where $L(c) \to 0$ as $c \to 0$ (see Proposition \ref{Prop}). Furthermore, \cite[Theorem 5.8]{CD} implies that $\mathcal{W}^{2}_2(\mu_t,\mu^{N}_t) \leq CN^{-1/2}$ and in addition, by triangle inequality, we deduce 
\begin{align}\label{a2}
L(c) \mathcal{W}^{2}_2(\mu^{N}_t,\mu_t^{\boldsymbol{X}^{N},N-1})  & \leq 2L(c) \left( \mathbb{E} \left[\sup_{t \in [0,T]} |X_t^{i} - X_t^{i,N}|^2 \right] +   \mathcal{W}^{2}_2(\mu_t^{\boldsymbol{X}^{N}},\mu_t^{\boldsymbol{X}^{N},N-1}) \right) \notag \\
&  \leq 2L(c) \mathbb{E} \left[\sup_{t \in [0,T]} |X_t^{i} - X_t^{i,N}|^2 \right] + CN^{-1/2},  
\end{align}
where we used $\mathcal{W}^{2}_2(\mu_t^{\boldsymbol{X}^{N}},\mu_t^{\boldsymbol{X}^{N},N-1}) \leq C N^{-1}$, which follows from \cite[Lemma 3.1]{LSAT}.

To summarise, taking \eqref{a1} and \eqref{a2} into account, we obtain 
\begin{align*}
\mathbb{E} \left[\sup_{t \in [0,T]} |X_t^{i} - X_t^{i,N}|^2 \right] & \leq C \Bigg( \mathbb{E} \left[\sup_{t \in [0,T]} |G(X^{i}_t,\mu_t)- G(X_t^{i,N},\mu_t^{\boldsymbol{X}^{N}})|^2 \right]  \\
& \quad + \mathbb{E} \left[\sup_{t \in [0,T]} |G(X_t^{i,N},\mu_t^{\boldsymbol{X}^{N}})- G(X_t^{i,N},\mu_t^{\boldsymbol{X}^{N},N-1})|^2 \right] + N^{-1/2} \Bigg).
\end{align*}
A similar analysis to above can be employed to handle the second term and hence, we arrive at
\begin{align}\label{a3}
\mathbb{E} \left[\sup_{t \in [0,T]} |X_t^{i} - X_t^{i,N}|^2 \right] & \leq   C \left(\mathbb{E} \left[\sup_{t \in [0,T]} |G(X^{i}_t,\mu_t)- G(X_t^{i,N},\mu_t^{\boldsymbol{X}^{N}})|^2 \right] + N^{-1/2} \right).
\end{align}
To further estimate \eqref{a3}, we apply It\^{o}'s formula to derive
\allowdisplaybreaks
\begin{align*}
& G(X^{i}_t,\mu_t)- G(X_t^{i,N},\mu_t^{\boldsymbol{X}^{N}})  \\
& = G(X^{i}_0,\mu_0)- G(X_0^{i,N},\mu_0^{\boldsymbol{X}^{N}}) \\
& \quad +  \int_{0}^{t} \left(b(X^{i}_s,\mu_s) + \alpha({\mu_s}) \bar{\phi}'(X^{i}_s) b(X^{i}_s,\mu_s) + \frac{1}{2} \alpha({\mu_s})  \bar{\phi}''(X^{i}_s) \sigma^2(X^{i}_s) \right) \, \mathrm{d}s  \\
& \quad  - \int_{0}^{t} \left(b(X^{i,N}_s,\mu_s^{\boldsymbol{X}^{N}}) + \alpha(\mu_s^{\boldsymbol{X}^{N}}) \bar{\phi}'(X^{i,N}_s) b(X^{i,N}_s,\mu_s^{\boldsymbol{X}^{N}}) + \frac{1}{2} \alpha(\mu_s^{\boldsymbol{X}^{N}})  \bar{\phi}''(X^{i,N}_s) \sigma^2(X^{i,N}_s) \right) \, \mathrm{d}s  \\
& \quad + \int_{0}^{t} \left(\sigma(X^{i}_s) + \alpha({\mu_s}) \bar{\phi}'(X^{i}_s)\sigma(X^{i}_s) - \sigma(X^{i,N}_s) - \alpha(\mu_s^{\boldsymbol{X}^{N}}) \bar{\phi}'(X^{i,N}_s)\sigma(X^{i,N}_s) \right) \, \mathrm{d}W^{i}_s \\
& \quad + \int_{0}^{t} \Big(  \partial_s G(X^{i}_s,\mu_s) -  \frac{1}{2N} \sum_{k =1}^{N} \partial_y \partial_{\mu}G(X_s^{i,N},\mu_s^{\boldsymbol{X}^{N}})(X_s^{k,N})  \sigma^2(X^{k,N}_s)  \\
& \hspace{2cm} -  \frac{1}{N} \sum_{k =1}^{N} \partial_{\mu} G(X_s^{i,N},\mu_s^{\boldsymbol{X}^{N}})(X_s^{k,N}) b(X^{k,N}_s,\mu_s^{\boldsymbol{X}^{N}})  \Big) \, \mathrm{d}s \\
& \quad - \int_{0}^{t} \frac{1}{N} \sum_{k =1}^{N} \partial_{\mu} G(X_s^{i,N},\mu_s^{\boldsymbol{X}^{N}})(X_s^{k,N}) \sigma(X^{k,N}_s) \, \mathrm{d}W_s^{k}  \\
& \quad - \int_{0}^{t} \frac{1}{2N^{2}} \sum_{k=1}^{N}\partial^2_{\mu}G(X_s^{i,N},\mu_s^{\boldsymbol{X}^{N}})(X_s^{k,N},X_s^{k,N}) \sigma^2(X_s^{k,N}) \, \mathrm{d}s  \\
& \quad - \int_{0}^{t} \frac{1}{N} \partial_{x_i}  \partial_{\mu}G(X_s^{i,N},\mu_s^{\boldsymbol{X}^{N}})(X_s^{i,N})\sigma^2(X_s^{i,N}) \, \mathrm{d}s =: \sum_{i=1}^{7} \Pi_i. \nonumber
\end{align*}
It is clear due to (H.\ref{Assum:AAA}(\ref{Assum:AAA2})) that $\mathbb{E}[|\Pi_6|^2] + \mathbb{E}[|\Pi_7|^2] \leq CN^{-1}$. For the term $\Pi_5$ we derive, using BDG's inequality,
\begin{align*}
\mathbb{E}[|\Pi_5|^2] & \leq   \frac{C}{N} \Bigg( \sum_{k=1}^{N} \mathbb{E}\left[ \int_{0}^{t}\left( \partial_{\mu} G(X_s^{i,N},\mu_s^{\boldsymbol{X}^{N}})(X_s^{k,N}) \sigma(X^{k,N}_s) - \partial_{\mu}\alpha(\mu_s)(X_s^{k})\bar{\phi}(X^{k}_s) \sigma(X^{k}_s) \right)^2 \, \mathrm{d}s \right] \\
& \quad + \frac{1}{N} \sum_{k,l=1}^{N} \mathbb{E}\left[ \left(\int_{0}^{t}\partial_{\mu}\alpha(\mu_s)(X_s^{k})\bar{\phi}(X^{k}_s) \sigma(X^{k}_s) \, \mathrm{d}W_s^{k} \right)  \left(\int_{0}^{t}  \partial_{\mu}\alpha(\mu_s)(X_s^{l})\bar{\phi}(X^{l}_s)  \sigma(X^{l}_s) \, \mathrm{d}W_s^{l} \right) \right] \Bigg).
\end{align*}
Therefore, taking the Lipschitz continuity of $(x,\mu) \mapsto  \partial_{\mu}\alpha(\mu)(x) \bar{\phi}(x) \sigma(x)$ and the independence of $(X_t^{k})_{k \in \lbrace 1, \ldots, N \rbrace}$, for $t \in [0,T]$, into account and using $\mathcal{W}_2(\mu_s^{\boldsymbol{X}^{N}},\mu_s) \leq \mathcal{W}_2(\mu_s^{\boldsymbol{X}^{N}},\mu^{N}_s) + \mathcal{W}_2(\mu_s^N,\mu_s)$ along with \cite[Theorem 5.8]{CD}, we obtain
\begin{align*}
\mathbb{E}[|\Pi_5|^2] & \leq C \left( \int_{0}^{T} \mathbb{E} \left[\sup_{s \in [0,t]} |X_s^{i} - X_s^{i,N}|^2 \right] \, \mathrm{d}t + N^{-1/2} \right).
\end{align*}

Similar to Lemma \ref{lem:Lip} combined with BDG's inequality and H\"{o}lder's inequality, we deduce
\begin{align*}
\mathbb{E}[|\Pi_1|^2]+  \mathbb{E}[|\Pi_2|^2]+ \mathbb{E}[|\Pi_3|^2] + \mathbb{E}[|\Pi_4|^2] & \leq C \left( \int_{0}^{T} \mathbb{E} \left[\sup_{s \in [0,t]} |X_s^{i} - X_s^{i,N}|^2 \right] \, \mathrm{d}t + N^{-1/2} \right).
\end{align*}
Therefore, inserting the estimates for $\mathbb{E}[|\Pi_1|^2], \ldots, \mathbb{E}[|\Pi_7|^2]$ back into \eqref{a3}, we deduce the claim using Gronwall's inequality. 
\end{proof}

The following hybrid explicit-implicit time-stepping algorithm computes a discrete-time approximation of $(X_t^{i,N})_{t \in [0,T]}$, denoted by $X_{t_n}^{i,N,M}$ for $n \in \lbrace 0, \ldots, M \rbrace$ and $i \in \lbrace 1, \ldots, N \rbrace$: 
\begin{itemize}
\item  Set $\tilde{X}^{i,N,M}_{t_0} = G(X^{i,N}_0,\mu_0^{\boldsymbol{X}^{N}})$ and $X^{i,N,M}_{t_0} = X^{i,N}_0= \xi^{i}$. 
\item  For $n \geq 1$, compute 
\begin{equation*}
\tilde{X}^{i,N,M}_{t_n} = \tilde{X}^{i,N,M}_{t_{n-1}} + B_i(\tilde{X}^{1,N,M}_{t_{n-1}}, \ldots, \tilde{X}^{N,N,M}_{t_{n-1}}) h  + \sum_{j=1}^{N}\Sigma^{i,j}(\tilde{X}^{1,N,M}_{t_{n-1}}, \ldots, \tilde{X}^{N,N,M}_{t_{n-1}}) \Delta W^{j}_n,
\end{equation*}
where $B_i$ and $\Sigma^{i,j}$ are defined by (\ref{Drift}) and (\ref{Diff}), respectively.
\item  Find $X^{i,N,M}_{t_n}$ such that $X^{i,N,M}_{t_n} = G^{-1}(\tilde{X}^{i,N,M}_{t_n},\mu_{t_n}^{\boldsymbol{X}^{N,M}})$, with $\mu_{t_n}^{\boldsymbol{X}^{N,M}}(\mathrm{d}x) = \frac{1}{N} \sum_{j=1}^{N} \delta_{X^{j,N,M}_{t_n}}(\mathrm{d}x)$.
\end{itemize} 

\begin{remark}
The implicit function theorem applied to the function 
\begin{equation*}
\boldsymbol{F}_N(\boldsymbol{x}^{N},\boldsymbol{y}^{N}) = \boldsymbol{y}^{N}- \left(G^{-1}(x_1,\mu^{\boldsymbol{y}^{N}}), \ldots, G^{-1}(x_N,\mu^{\boldsymbol{y}^{N}}) \right)^{\top}, \quad \boldsymbol{x}^{N}, \boldsymbol{y}^{N} \in \RR^{N},
\end{equation*}
implies that we can express $\boldsymbol{y}^{N}$ in terms of $\boldsymbol{x}^{N}$. The applicability of the implicit function theorem follows from similar arguments to the ones presented in Lemma \ref{lem:Inverse} along with Proposition \ref{Prop:LD}.
\end{remark}

\begin{remark}
We could also define an explicit scheme by setting $X^{i,N,M}_{t_n} = G^{-1}(\tilde{X}^{i,N,M}_{t_n},\mu_{t_n}^{\tilde{\boldsymbol{X}}^{N,M}})$, with $\mu_{t_n}^{\tilde{\boldsymbol{X}}^{N,M}}(\mathrm{d}x) = \frac{1}{N} \sum_{j=1}^{N} \delta_{\tilde{X}^{j,N,M}_{t_n}}(\mathrm{d}x)$. However, to derive a strong convergence rate for the resulting scheme, one has to analyse the quantity $\mathbb{E}[|X^{i,N,M}_{t_n}-\tilde{X}^{i,N,M}_{t_n}|^2]$. Similar arguments as for Scheme 2 could possibly be used here, but our current analysis does not allow us to derive an optimal convergence rate in $h$. 
\end{remark}

\begin{theorem}\label{TH:E3}
Let Assumption (H.\ref{Assum:AAA}) hold, let $\xi \in L_p^{0}(\RR)$ for $p > 4$ and assume the mapping $\mathbb{R} \setminus \lbrace 0 \rbrace \times \mathcal{P}_2(\RR) \ni (x,\mu) \mapsto b(x,\mu)$ to be uniformly bounded. Let $(X_t^{i})_{t \in [0,T]}$ be the unique strong solution of (\ref{eq:Model2}) driven by the Brownian motion $(W_t^{i})_{t \in [0,T]}$ with initial data $\xi^{i}$, and $X_{t_n}^{i,N,M}$ for  $n \in \lbrace 0, \ldots, M\rbrace$ be defined by the above algorithm. Then, there exists a constant $C>0$ (independent of $N$ and $M$) such that 
\begin{equation*}
\max_{i \in \lbrace 1, \ldots, N \rbrace} \max_{n \in \lbrace 0, \ldots, M \rbrace} \mathbb{E}\left[ |X_{t_n}^{i} - X_{t_n}^{i,N,M}|^2 \right] \leq C (N^{-1/2} + h).
\end{equation*}
\end{theorem}
\begin{proof}
We start with the observation that for any $n \in \lbrace 0, \ldots, M \rbrace$
\begin{align*}
|X_{t_n}^{i,N} - X_{t_n}^{i,N,M}|^2 & = |G^{-1}(G(X_{t_n}^{i,N},\mu_{t_n}^{\boldsymbol{X}^{N},N-1}),\mu_{t_n}^{\boldsymbol{X}^{N},N-1})-G^{-1}(\tilde{X}^{i,N,M}_{t_n},\mu_{t_n}^{\boldsymbol{X}^{N,M}})|^2 \\
& \leq  2|G^{-1}(G(X_{t_n}^{i,N},\mu_{t_n}^{\boldsymbol{X}^{N},N-1}),\mu_{t_n}^{\boldsymbol{X}^{N},N-1})-G^{-1}(G(X_{t_n}^{i,N},\mu_{t_n}^{\boldsymbol{X}^{N}}),\mu_{t_n}^{\boldsymbol{X}^{N}})|^2 \\
& \quad + 2|G^{-1}(G(X_{t_n}^{i,N},\mu_{t_n}^{\boldsymbol{X}^{N}}),\mu_{t_n}^{\boldsymbol{X}^{N},N})-G^{-1}(\tilde{X}^{i,N,M}_{t_n},\mu_{t_n}^{\boldsymbol{X}^{N,M}})|^2.
\end{align*}
Using the arguments from the previous lemma, the first term can be shown to be of order $\mathcal{O}(N^{-1})$. For the second term, we derive the estimate 
\begin{align*}
& |G^{-1}(G(X_{t_n}^{i,N},\mu_{t_n}^{\boldsymbol{X}^{N}}),\mu_{t_n}^{\boldsymbol{X}^{N}})-G^{-1}(\tilde{X}^{i,N,M}_{t_n},\mu_{t_n}^{\boldsymbol{X}^{N,M}})|^2 \\
& \leq C \left( |G(X_{t_n}^{i,N},\mu_{t_n}^{\boldsymbol{X}^{N}}) - \tilde{X}^{i,N,M}_{t_n}|^2 + L(c)\mathcal{W}_2^{2}(\mu_{t_n}^{\boldsymbol{X}^{N}},\mu_{t_n}^{\boldsymbol{X}^{N,M}}) \right),
\end{align*}
where $L(c) \to 0$ as $c \to 0$, as in Proposition \ref{Prop}.

Therefore, we get
\begin{align*}
\mathbb{E} \left[|X_{t_n}^{i,N} - X_{t_n}^{i,N,M}|^2 \right] \leq C \left( N^{-1} +  \mathbb{E} \left[|G(X_{t_n}^{i,N},\mu_{t_n}^{\boldsymbol{X}^{N}}) - \tilde{X}^{i,N,M}_{t_n}|^2 \right] \right). 
\end{align*}
Using now the definition of $\tilde{X}^{i,N,M}_{t_n}$ and Lemma \ref{lem:lip}, one shows that there exists a constant $C>0$ such that $\mathbb{E} \left[|G(X_{t_n}^{i,N},\mu_{t_n}^{\boldsymbol{X}^{N}}) - \tilde{X}^{i,N,M}_{t_n}|^2 \right] \leq Ch$. This together with Lemma \ref{PCLem} yields the claim.
\end{proof}
\section{Numerical illustration}\label{SEC:NUM}
In the following, we will present two examples of McKean--Vlasov SDEs and interacting particle systems exhibiting discontinuous drifts in order to motivate the theoretical study of such equations and to numerically illustrate the strong convergence behaviour of an Euler--Maruyama scheme. These models find applications in biology and mathematical finance, in particular systemic risk. 

As we do not know the exact solution of the considered equations, the convergence rate (in terms of number of time-steps) was determined by comparing two solutions computed on a fine and coarse grid, respectively, 
where the same Brownian motions were used.
In order to illustrate the strong convergence behaviour in the uniform time-step $h$, we compute the root-mean-square error (RMSE) by comparing the numerical solution at level $l$ of the time discretisation with the solution at level $l-1$ at the final time $T=1$. 
To be precise, as error measure we use the quantity
\begin{equation*}
\text{RMSE}:= \sqrt{\frac{1}{N} \sum_{i=1}^{N} \left(X_T^{i,N,M_l} - X_T^{i,N,M_{l-1}} \right)^2},
\end{equation*}
where $M_l = 2^lT$ and by $X_T^{i,N,M_l}$ we denote the approximation of $X$ at time $T$ computed based on $N$ particles and $2^lT$ time-steps. The number of particles used in the tests will be specified below. 
\subsection{Neuronal interactions}
In this section, we provide a numerical illustration for a specific model for neuronal interactions.
Interacting particle systems are ubiquitous in neuroscience, such as the Hodgkin-Huxley model \cite{BFFT,BO2} or mean-field equations describing the behaviour of a (large) network of interacting spiking neurons \cite{DINRT}.
For other mean-field models appearing in neuroscience, we refer to the references given in \cite{DINRT}.

A recent model of the action potential of neurons is described in \cite{FPZ} and involves discontinuous coefficients. The reason for the necessity of introducing discontinuities is the following: After a charging phase of an individual neuron, subject to spikes of nearby neurons, randomness, and the effect of discharge with constant rate, the neuron emits a spike to the network once a certain threshold is hit and is then in a recovery phase.
The change between these two phases is characterised by a discontinuity in the dynamics describing the potential of each neuron. 

The action potential of $N$ interacting neurons at time $t \in [0,T]$, $V^{i,N}_t (\text{mod } 2) \in [0,2)$, where $V^{i,N}_t \in \RR$, $i \in \lbrace 1, \ldots, N \rbrace$, is modelled by the discontinuous mean-field equation
\begin{align*}
\mathrm{d}V^{i,N}_t &= \lambda(V^{i,N}_t (\text{mod } 2))  \, \mathrm{d}t + \sigma^{\varepsilon}(V^{i,N}_t (\text{mod } 2))  \, \mathrm{d}W^{i}_t \\
& \quad + \frac{1}{N}\sum_{j=1}^N \Theta(\xi_i,\xi_j) \mathrm{I}_{[1,1+\kappa]}(V^{j,N}_t (\text{mod } 2)) \mathrm{I}_{[0,1]}(V^{i,N}_t (\text{mod } 2))  \, \mathrm{d}t,
\end{align*}
with square-integrable random initial values $V_0^{i,N} = \eta_i$, for $i \in \lbrace 1, \ldots, N \rbrace$ and $0 < \kappa < 1$ fixed. 
The set of i.~i.~d.\ random variables $\lbrace \xi_1, \ldots, \xi_N \rbrace$, $\xi_i \in D$,
describes the location of the $N$ (non-moving) neurons, where $D$ is modelled as an open connected domain of $\RR^3$. Hence, the position of each neuron is given by $X_t^{i,N} = \xi^{i}$ at each time $t \in [0,T]$. It is further assumed that $\lbrace \xi_1, \ldots, \xi_N, \eta_1, \ldots, \eta_N \rbrace $ are independent for each integer $N \geq 1$.
The standard Brownian motions $(W_t^{i})_{t \in [0,T]}$ are independent, and independent of $\xi_i$ and $\eta_i$, for  $i \in \lbrace 1, \ldots, N \rbrace$. Observe that $V^{i,N}$ is specified by an SDE, where the drift has a discontinuity in the state variable (due to the choice of $\lambda$; see below for details), but also has a jump, in case a particle $j \neq i$ reaches the the critical values $1$ or $1+\kappa$.

The following conditions are imposed in \cite{FPZ} to guarantee strong well-posedness of the particle system above (see, \cite[Theorem 2.2]{FPZ}) and the associated McKean--Vlasov equation (see, \cite[Theorem 6.1]{FPZ}); propagation of chaos type results, i.e., weak convergence of the law of the empirical distribution of $(\xi^{i},V^{i,N})$ to a Dirac measure centred at the law of the solution to the underlying McKean--Vlasov equation, are shown in \cite[Theorem 5.7]{FPZ}: 
\begin{enumerate}
\item $\lambda(v) = -\hat{\lambda}v \mathrm{I}_{[0,1]}(v) + \mathrm{I}_{(1,2)}, \quad \hat{\lambda} >0$, \quad  $\Theta(x,y) = \sin(|x-y|)$, for $x,y \in \RR^3$;
\item $ \sigma^{\varepsilon}$ is a $\mathcal{C}_b^{1}([0,2])$ function satisfying $ \sigma^{\varepsilon} \geq \sqrt{2 \varepsilon} > 0$ and
\begin{align*}
\sigma^{\varepsilon}(v) = \sqrt{2\varepsilon} \text{ on } [1,2], \quad \sigma^{\varepsilon}(2) =  \sigma^{\varepsilon}(0) = \sqrt{2 \varepsilon},  \quad (\sigma^{\varepsilon})'(0) = (\sigma^{\varepsilon})'(2) = 0, \text{ with } \varepsilon>0 \text{ fixed}.
\end{align*} 
\end{enumerate} 

\noindent
These conditions specify our Example 1. For Example 2, the diffusion term was chosen $\sigma^{\varepsilon}(x) = \sqrt{2\varepsilon}+x$. For our tests, we used $N=10^3$. Furthermore, we set $\kappa =0.01$, $\hat{\lambda} = 0.02$ and $\varepsilon = 0.1$. The initial values $\eta_1, \ldots, \eta_N$ are chosen as independent normal random variables with mean $1$ and standard deviation $2$. Also, these values are considered modulo 2. The variables $\xi_1, \ldots, \xi_N$ are independent three-dimensional random variables chosen from the same multivariate normal distribution with some given mean vector and covariance matrix.

We investigate numerically the convergence of Scheme 2, i.e., the Euler--Maruyama scheme without applying any transformations.
In Fig.\ \ref{fig:StrongOrder1}, we observe strong convergence of order $3/4$ for Example 1, which is most likely due to the choice of $\sigma^{\varepsilon} = \sqrt{2\varepsilon}$ as constant. In \cite{MGLY219} a Milstein scheme for one-dimensional SDEs with discontinuities in the drift was derived and a strong convergence order of $3/4$ was proven. In addition, it is conjectured in \cite{MGLY219} (Conjecture 1 and Conjecture 2) that the rate $3/4$ is optimal. 
\begin{figure}[!h]
\centering
\includegraphics[width=0.5\textwidth]{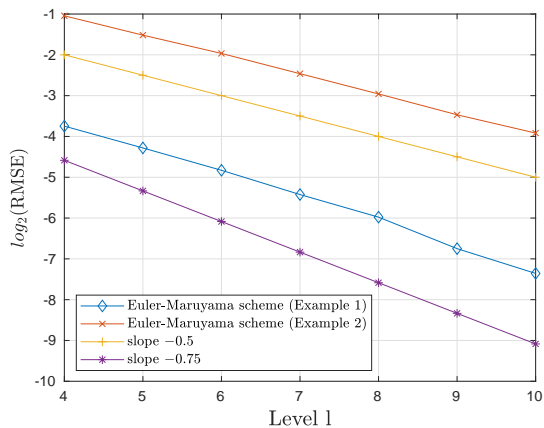}
\caption{Strong convergence of the Euler--Maruyama scheme applied to the particle system obtained by approximating the equation for the action potential of the neurons.}
\label{fig:StrongOrder1}
\end{figure}

\subsection{Systemic risk}
In this section, we consider a McKean--Vlasov SDE of the form
\begin{equation}\label{eq:control}
\mathrm{d}X_t = \left( a \left(\mathbb{E}[X_t]- X_t \right) + \kappa_1 \mathrm{I}_{ \lbrace X_t \leq 0 \rbrace} + \kappa_2 \mathrm{I}_{\lbrace X_t > 0 \rbrace}  \right) \, \mathrm{d}t + (\sigma + X_t) \, \mathrm{d}W_t, \quad X_0=x \in \RR,
\end{equation}
where $a \geq0$ is the mean-reversion rate, $\kappa_1 < 0$, $\kappa_2>0$ and $\sigma >0$. The strong well-posedness of (\ref{eq:control}) follows from Proposition \ref{Prop1Ex}. This equation can be linked to a model of systemic risk in \cite{CARM}, where a mean-field game of $N$
 banks borrowing from, and lending to, a central bank is proposed.
 The banks control the rate of their borrowing depending on their (log-)monetary reserves, which are modelled by a system of SDEs
 with interaction through their average. In this setting flocking, and thus systemic default events, may occur.
 Here, we give a slight reformulation of this problem following \cite[Section 4]{XGUO}. The problem consists in finding $(\hat{\mu}_t,\hat{\beta}_t)_{t \in [0,T]}$, where $(\hat{\mu}_t)_{t \in [0,T]}$ is a flow of measures in $\mathcal{C}([0,T],\mathcal{P}_2(\RR))$ and $(\hat{\beta}_t)_{t \in [0,T]}$ is an adapted, square-integrable control process (describing the rate of borrowing from, or lending to, the central bank), such that $(\hat{\beta}_t)_{t \in [0,T]}$ minimises 
the objective function given by
\begin{equation*}
J^{\hat{\mu}}(x;\beta) = \mathbb{E} \left[ \int_{0}^{T}\left( r |\beta_t| + \frac{\varepsilon}{2} \left(X^{\hat{\mu},\beta}_t - \int_{\RR} x \, \hat{\mu}_t(\mathrm{d}x) \right)^2 \right) \, \mathrm{d}t  + \frac{c}{2} \left(X^{\hat{\mu},\beta}_T - \int_{\RR} x \, \hat{\mu}_T(\mathrm{d}x)  \right)^2  \right],
\end{equation*}   
for $r, \varepsilon, c \geq 0$, where
\begin{equation}\label{eg:Mod1}
\mathrm{d}X^{\hat{\mu},\beta}_t = \left( a \left(\int_{\RR} x \, \hat{\mu}_t(\mathrm{d}x)- X^{\hat{\mu},\beta}_t \right) + \beta_t \right) \, \mathrm{d}t + (\sigma + X^{\hat{\mu},\beta}_t) \, \mathrm{d}W_t, \quad X^{\hat{\mu},\beta}_0=x \in \RR,
\end{equation}
and $\hat{\mu}_t = \mathcal{L}_{X^{\hat{\mu},\hat{\beta}}_t}$ for all $t \in [0,T]$.

For simplicity, we did not add the common noise term as in \cite{CARM}. In addition, we modified the diffusion to allow it to be degenerate. 
In \cite{XGUO}, a constraint $\beta_t \in [\kappa_1,\kappa_2]$ on the borrowing/lending rate is imposed for all $t \in [0,T]$. The minimiser of the objective function for this mean-field game (with constant diffusion term) is shown to be a control of bang-zero-bang type 
(see \cite[equation (24)]{XGUO} for an analytic expression).
Written in feedback form, this optimal control strategy has discontinuities  that are time-dependent, i.e., the zero-control region changes over time,
a setting not covered by our current analysis.

Using instead the special bang-bang type control $(\beta^{*}_t)_{t \in [0,T]}$ of the form 
\begin{align*}
\beta^{*}_t =
\begin{cases}
\kappa_1, &\text{ if } X_t \leq 0 \\
\kappa_2, & \text{ if } X_t > 0, \\
\end{cases}
\end{align*}
and plugging $\beta_t^{*}$ back into \eqref{eg:Mod1} results into an equation of the form (\ref{eq:control}) with a one-point discontinuity.
  
In our numerical experiments, we set $\kappa_1= -0.5$, $\kappa_2=0.5, x=0$ and $\sigma =0.7$. Further, we consider three different choices for the mean-reversion rate, i.e., we set $a=1,5,10$. The expected value is approximated by the empirical mean of $N=10^4$ particles. For a larger value of $a$, we expect sample paths of (\ref{eq:control}) to be more concentrated around the mean of the samples; see Fig.\ \ref{fig:StrongOrder3a} and \ref{fig:StrongOrder3b} (with $T=1$ and $M=2^7$). We can also observe that for a stronger concentration effect the strong approximation behaviour becomes better; see Fig.\ \ref{fig:StrongOrder3b}. 

\begin{figure}[H]
\centering
\begin{subfigure}[b]{0.45\textwidth}
\includegraphics[width=\textwidth]{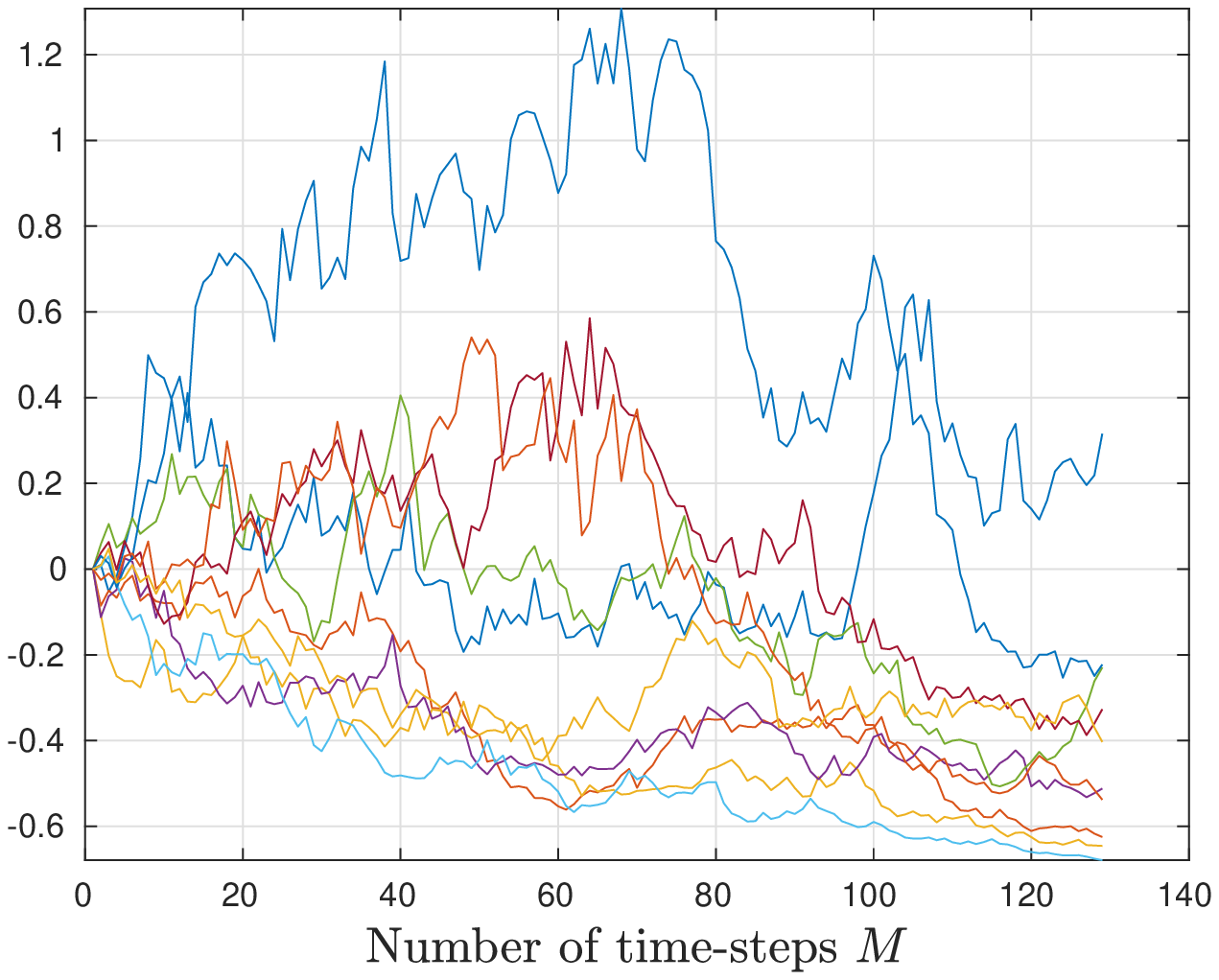}
\caption{}
\label{fig:StrongOrder3a}
\end{subfigure}
\begin{subfigure}[b]{0.45\textwidth}
\includegraphics[width=\textwidth]{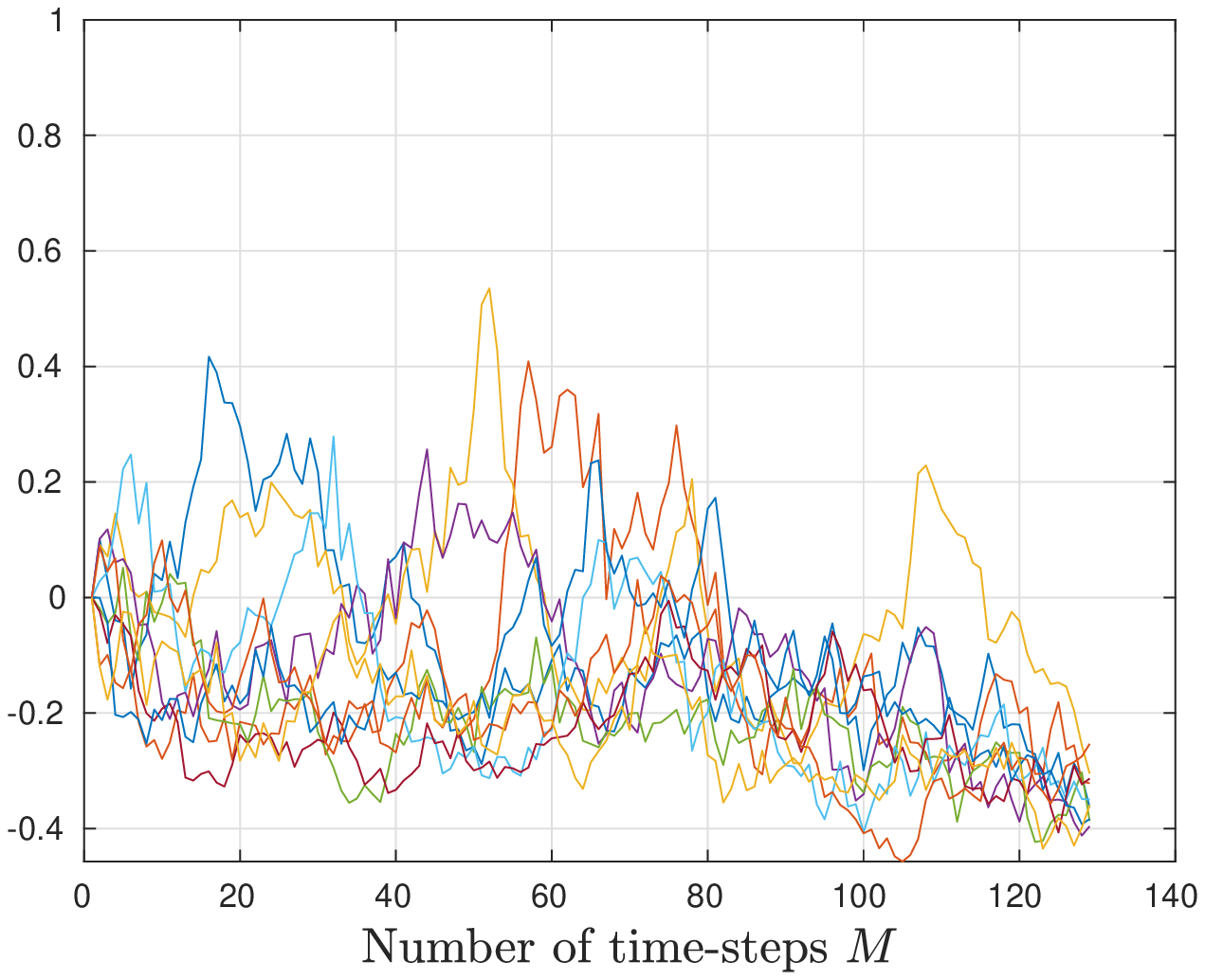}
\caption{}
\label{fig:StrongOrder3b}
\end{subfigure}
\caption{Sample trajectories of the particle system associated with (\ref{eq:control}) for $N=10$ and $a=1$ (left) and $a=10$ (right).}
\end{figure}

\begin{figure}[H]
\centering
\includegraphics[width=0.5\textwidth]{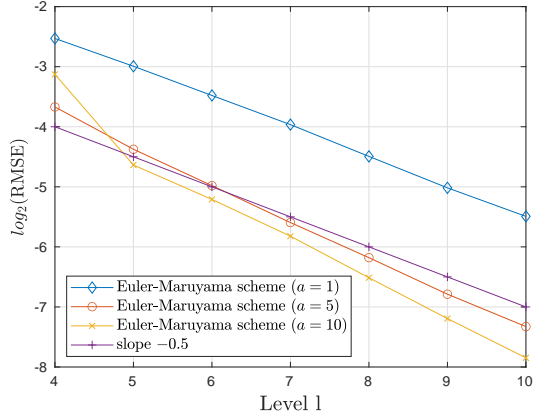}
\caption{Strong convergence of the Euler--Maruyama scheme applied to the particle system obtained by approximating the equation (\ref{eq:control}) with $a=1,5,10$.}
\label{fig:StrongOrder2}
\end{figure} 

\appendix
\section{Auxiliary results}
\subsection{$L$-differentiability of $G^{-1}$}\label{SEC:AA1}
\begin{proposition}\label{Prop:LD}
Let $G\colon \RR\times \mathcal{P}_2(\RR) \to \RR$ be defined by \eqref{eq:TransMeas2} with $c>0$ as in Remark \ref{REM:LIP} and let the assumptions of Proposition \ref{Prop} be satisfied. Assume that $G$ is $L$-differentiable, and $\RR \ni y \mapsto \partial_{\mu}G(x,\mu)(y)$ is in $\mathcal{C}^{1}(\RR,\RR)$ for all $(x,\mu) \in \RR \times \mathcal{P}_2(\RR)$. Then, the inverse $G^{-1}\colon \RR\times \mathcal{P}_2(\RR) \to \RR$ is continuously $L$-differentiable with 
\begin{equation}\label{eq:partial-mu-G1}
\partial_{\mu}G^{-1}(x,\mu)(y)= \frac{-\partial_{\mu} G(G^{-1}(x,\mu),\mu)(y)}{\partial_{x} G(G^{-1}(x,\mu),\mu)},
\end{equation}
for all $(x,\mu,y) \in \RR \times \mathcal{P}_2(\RR) \times \RR$.
\end{proposition}
\begin{proof}
First, we remark that (H.\ref{Assum:AA}(\ref{Assum:AA3})) guarantees the $L$-differentiability of $G$.
Let now $X,Y \in L_2(\Omega, \mathcal{F},\mathbb{P};\mathbb{\RR})$ with
$\mathcal{L}_{X}=\mu$ and $\mathcal{L}_{X+Y}=\nu$, for $\mu, \nu \in
\mathcal{P}_2(\RR)$. Consider the lifted inverse function $\tilde{G}^{-1}$
defined by $\tilde{G}^{-1}(y,X):= G^{-1}(y,\mu)$, for any $y \in \RR$.
Similarly, we set $\tilde{G}(x,X):= G(x,\mu)$ and $\tilde{G}(x,X+Y):= G(x,\nu)$, for any $x \in \RR$.
Now fix $y\in \RR$ and set $x:=\tilde{G}^{-1}(y,X)$, $h:=\tilde{G}^{-1}(y,X+Y)-x$. 

Observe that $y=G(G^{-1}(y,\nu),\nu)=G(x+h,\nu)$ and also $y=G(G^{-1}(y,\mu),\mu)=G(x,\mu)$. In addition, due to the Lipschitz continuity of $\mu \mapsto G^{-1}(y,\mu)$, we have 
\begin{equation}\label{eq:help1}
|h|  = |G^{-1}(y,\nu) - G^{-1}(y,\mu)| \leq L \mathcal{W}_2(\nu,\mu) \leq L \|Y \|_{L_2}.
\end{equation}
In what follows, we aim to show \eqref{eq:partial-mu-G1}: 
\begin{align*}
& \frac{\left| \tilde{G}^{-1}(y,X+Y) - \tilde{G}^{-1}(y,X) - \left \langle \frac{-\partial_{\mu} G(x,\mu)(X)}{\partial_{x} G(x,\mu)},Y \right \rangle_{L_2} \right|}{ \|Y \|_{L_2}}  \leq  C \frac{\left| -\partial_{x} G(x,\mu) h  - \left \langle \partial_{\mu} G(x,\mu)(X),Y \right \rangle_{L_2} \right|}{ \|Y \|_{L_2}},
\end{align*}
for some constant $C>0$, where we used the boundedness of $x \mapsto \partial_{x} G(x,\mu)$. 
Employing this estimate, along with the identity
$\tilde{G}(x,X)=\tilde{G}(x+h,X+Y)$ and \eqref{eq:help1}, we obtain
\begin{align*}
& \frac{\left| \tilde{G}^{-1}(y,X+Y) - \tilde{G}^{-1}(y,X) - \left \langle \frac{-\partial_{\mu} G(x,\mu)(X)}{\partial_{x} G(x,\mu)},Y \right \rangle_{L_2} \right|}{ \|Y \|_{L_2}}  \\
&\leq C\frac{\left| -\partial_{x} G(x,\mu)h  + \tilde{G}(x+h,X+Y)  -  \tilde{G}(x,X) - \left \langle \partial_{\mu} G(x,\mu)(X),Y \right \rangle_{L_2} \right|}{ \|Y \|_{L_2}} \\
&\leq C\frac{\left|\tilde{G}(x+h,X+Y) - \tilde{G}(x,X+Y) -\partial_{x} G(x,\nu)h  \right|}{ \|Y \|_{L_2}}
+C\frac{\left|\partial_{x} G(x,\nu)h  -\partial_{x} G(x,\mu)h   \right|}{ \|Y \|_{L_2}} \\
&\quad+C\frac{\left| \tilde{G}(x,X+Y)  -  \tilde{G}(x,X) - \left \langle \partial_{\mu} G(x,\mu)(X),Y \right \rangle_{L_2} \right|}{ \|Y \|_{L_2}}\\
 &\leq CL\frac{\left|G(x+h,\nu) - G(x,\nu) -\partial_{x} G(x,\nu)h  \right|}{ |h|}
+CL\left|\partial_{x} \tilde G(x,X+Y)  -\partial_{x} \tilde G(x,X)   \right| \\
&\quad+C\frac{\left| \tilde{G}(x,X+Y)  -  \tilde{G}(x,X) - \left \langle \partial_{\mu} G(x,\mu)(X),Y \right \rangle_{L_2} \right|}{ \|Y \|_{L_2}} \,.
\end{align*}
Now, if $\| Y \|_{L_2} \to 0$, then also $ |h| \to 0$, so the terms in the last estimate tend to zero from which the claim follows. Furthermore, we remark that the mapping $y \mapsto \frac{-\partial_{\mu} G(G^{-1}(x,\mu),\mu)(y)}{\partial_{x} G(G^{-1}(x,\mu),\mu)}$ is in $\mathcal{C}^{1}(\RR,\RR)$.
\end{proof}
\subsection{The class $\mathscr{C}$}\label{SEC:APP2}
For a given $ N \in \NN$, we define for all $k\in \{1,\ldots, N\}$ the sets
\[
\Theta^k:=\{\boldsymbol{x}^{N}=(x_1,\ldots,x_N)^{\top} \colon x_k=0\}.
\] 

\begin{definition}
For $ N \in \NN$, we define 
the class $\mathscr{C}$ of functions 
$\boldsymbol{G}_N\colon \RR^N\to \RR^N$ with the following properties:
\begin{enumerate}[(p1)]
\item\label{it:c1} $\boldsymbol{G}_N$ is $\mathcal{C}^1(\RR^N,\RR^{N \times N})$;
\item\label{it:detne0} for all $ \boldsymbol{x}^{N} \in \RR^N$, $\det\big(\boldsymbol{G}_N'(\boldsymbol{x}^{N})\big)\ne 0$;
\item\label{it:hadamard} $\lim_{|\boldsymbol{x}^{N}|\to \infty} |\boldsymbol{G}_N(\boldsymbol{x}^{N})|=\infty$;
\item\label{it:theta-zero} for all $k\in \{1,\dots,N \}$ and all $\boldsymbol{x}^{N} \in \Theta^k$, we have $G_k(\boldsymbol{x}^{N})=0$;
\item\label{it:partial-delta} for all $j,k\in \{1,\dots,N \}$ and
all $\boldsymbol{x}^{N} \in \Theta^k$, we have $\partial_{x_j} G_k(\boldsymbol{x}^{N})=\delta_{j,k}$;  
\item\label{it:partial-mixed} for all $i,j,k\in \{1,\dots,N\}$ with $i\ne j$ the
mixed partial derivative $\partial_{x_i}\partial_{x_j}G_k$ exists
and is continuous;
\item\label{it:partial-second-neq} for all $j,k\in \{1,\dots,N \}$ with $j\ne k$ the
second partial derivative $\partial_{x_j}^2G_k$ exists
and is continuous;
\item\label{it:partial-second-eq} for all $k\in \{1,\dots,N\}$ the
second partial derivative $\partial_{x_k}^2 G_k$ exists 
 on $\RR^N\setminus \Theta^k$ and is continuous there.
\end{enumerate}
\end{definition}

The aim of this appendix is to prove the following theorem:

\begin{theorem}
If $\boldsymbol{G}_N \in \mathscr{C}$, then $\boldsymbol{G}_N$ is invertible with $\boldsymbol{G}_N^{-1}\in \mathscr{C}$.  
\end{theorem}
\begin{proof}
Let $\boldsymbol{G}_N = (G_1, \ldots, G_N)^{\top} \in \mathscr{C}$.
The Hadamard global inverse function theorem states that under assumptions
\pref{it:c1}, \pref{it:detne0}, and \pref{it:hadamard}, $\boldsymbol{G}_N$ is invertible with inverse in $\mathcal{C}^1(\RR^N,\RR^{N \times N})$. 
We denote $\boldsymbol{H}_N=\boldsymbol{G}^{-1}_N$ and conclude
\begin{itemize}
\item[\pref{it:c1}] $\boldsymbol{H}_N \in \mathcal{C}^1(\RR^N,\RR^{N \times N})$;
\item[\pref{it:detne0}] for all $\boldsymbol{x}^{N} \in \RR^d$, $\det\big(\boldsymbol{H}_N'(\boldsymbol{x}^{N})\big)\ne 0$;
\item[\pref{it:hadamard}] $\lim_{|\boldsymbol{x}^{N}|\to \infty} |\boldsymbol{H}_N(\boldsymbol{x}^{N})|=\infty$.
\end{itemize}

Let $k\in \{1,\ldots,N\}$ 
and fix values $x_1,\ldots,x_{k-1},x_{k+1},\ldots,x_N\in \RR$.
As $\boldsymbol{G}_N$ has a global inverse, the mapping $ \RR \ni x_k \mapsto G_k(x_1,\ldots,x_N)$ is invertible, i.e., there is precisely one $x_k$ with $G_k(x_1,\ldots,x_N)=0$.
On the other hand, $G_k(x_1,\ldots,x_{k-1},0,x_{k+1},\ldots,x_N)=0$, and consequently we proved:
\begin{itemize}
\item[\pref{it:theta-zero}] for all $k\in \{1,\dots,N\}$ and all $\boldsymbol{y} \in \Theta^k$, we have $H_k(\boldsymbol{y})=0$.
\end{itemize}

We note that, since $\boldsymbol{H}_N \circ \boldsymbol{G}_N$ is the identity on $\RR^N$
\begin{equation*}
\partial_{x_j}(H_k \circ \boldsymbol{G}_N)=\delta_{j,k},
\end{equation*}
and therefore, using \pref{it:partial-delta} for $\boldsymbol{G}_N$, we have for $\boldsymbol{x}^{N} \in \Theta^k$
\[
\delta_{j,k}
=\sum_{l=1}^N \partial_{y_l} H_k\circ \boldsymbol{G}_N(\boldsymbol{x}^{N})\,\partial_{x_j} G_l(\boldsymbol{x}^{N})
=\sum_{l=1}^N \partial_{y_l} H_k\circ \boldsymbol{G}_N(\boldsymbol{x}^{N})\,\delta_{j,l}=\partial_{y_j} H_k\circ \boldsymbol{G}_N(\boldsymbol{x}^{N}).
\]
Now if $\boldsymbol{y}^{N} \in \Theta^k$, then $\boldsymbol{H}_N(\boldsymbol{y}^{N})\in \Theta^k$, and therefore
\(
\partial_{y_j}H_k(y)=\partial_{y_j}H_k\circ \boldsymbol{G}_N(\boldsymbol{H}_N(\boldsymbol{y}^{N}))
=\delta_{j,k}.
\) To summarise, we have shown:
\begin{itemize}
\item[\pref{it:partial-delta}] for all $j,k\in \{1,\dots,N\}$ and all $\boldsymbol{y}^{N} \in \Theta^k$, we have $\partial_{y_j}H_k(\boldsymbol{y}^{N})=\delta_{j,k}$.
\end{itemize}

To prove (p\ref{it:partial-mixed})-(p\ref{it:partial-second-eq}), we first note that for any $j,k\in \{1,\dots,N\}$
\begin{equation*}
\partial_{y_j} G_k(\boldsymbol{H}_N(\boldsymbol{y}^{N})) = \sum_{l=1}^{N} \partial_{x_l} G_k(\boldsymbol{H}_N(\boldsymbol{y}^{N})) \partial_{y_j} H_l(\boldsymbol{y}^{N}) = \delta_{j,k},
\end{equation*}
and hence 
\begin{align*}
\left(\partial_{y_j} H_l(\boldsymbol{y}^{N}) \right)_{l \in \lbrace 1, \ldots, N \rbrace} = \left((\boldsymbol{G}'_N)^{-1}(\boldsymbol{H}_N(\boldsymbol{y}^{N})) \right)_j,
\end{align*}
where the subindex denotes the $j$-th column of $(\boldsymbol{G}'_N)^{-1}$. The higher order regularity properties of $\boldsymbol{H}_N$ follow from this expression and the second order differentiability properties of $\boldsymbol{G}_N$.

\end{proof}



\subsection*{Acknowledgment}
GL is supported by the Austrian
Science Fund (FWF): Project F5508-N26, which is part of the Special Research
Program `Quasi-Monte Carlo Methods: Theory and Applications'. WS is supported by a special Upper Austrian Government grant.
We want to express our gratitude to two anonymous referees for valuable comments, suggestions and improvements on an  earlier version of this article.
WS wants to thank Yufei Zhang for several helpful discussions on this topic.
\end{document}